\documentclass{amsart}    
\usepackage{amsthm, amsmath, amscd, amssymb}
\theoremstyle{plain}
\newtheorem{theorem}{Theorem}[section]
\newtheorem{lemma}[theorem]{Lemma}
\newtheorem{proposition}[theorem]{Proposition}
\newtheorem{corollary}[theorem]{Corollary}

\theoremstyle{remark}
\newtheorem{remark}[theorem]{Remark}

\newtheorem*{convention}{Convention}

\theoremstyle{definition}

\usepackage[mathscr]{eucal}
\usepackage{graphics, graphpap}
\usepackage{array, tabularx, longtable}
\numberwithin{equation}{section}

\def\ACVF{\mathrm{ACVF}}
\def\VF{\mathrm{VF}}

\def\RV{\mathrm{RV}}
\def\RES{\mathrm{RES}}

\def\hl{\mathbf h}
\def\fl{\mathrm{\dagger}}

\def\rv{\mathrm{rv}}
\def\val{\mathrm{val}}
\def\res{\mathrm{res}}
\def\Var{\mathrm{Var}}

\def\loc{\mathrm{loc}}

\def\pr{\mathrm{pr}}

\def\Gr{\text{Gr}}

\def\Spf{\text{Spf}}

\def\sing{\text{sing}}

\def\m{\mathfrak m}

\def\ord{\mathrm{ord}}
\def\Jac{\mathrm{Jac}}
\def\alg{\mathrm{alg}}

\def\Spec{\mathrm{Spec}}
\def\Spf{\text{Spf}}

\def\bdd{\mathrm{bdd}}
\def\vol{\mathrm{vol}}

\def\Ob{\mathrm{Ob}}

\def\X{\mathcal{X}}
\def\Y{\mathcal{Y}}
\def\Z{\mathcal{Z}}

\def\ac{\text{ac}}
\def\Lbb{\mathbb{L}}

\def\Gr{\mathrm{Gr}}
\def\gcd{\mathrm{gcd}}

\def\sing{\text{sing}}
\def\Hom{\mathrm{Hom}}

\def\k{\mathbf{k}}

\def\x{\mathbf{x}}
\def\y{\mathbf{y}}
\def\L{\mathbb{L}}
\def\A{\mathbb{A}}

\title[Regular and formal motivic Thom-Sebastiani theorem]{\bf The motivic Thom-Sebastiani theorem for regular and formal functions}  
\author{L\^e Quy Thuong}
\address{Institut de Recherche Math\'ematiques de Rennes (IRMAR), 263 Avenue du General Leclerc CS 74205 35042 Rennes Cedex, France {\rm(current)}}
\email{ thuong.lequy@univ-rennes1.fr}
\address{Department of Mathematics, Vietnam National University, 334 Nguyen Trai Street, Thanh Xuan District, Hanoi, Vietnam}
\email{thuonglq@vnu.edu.vn}

\begin{document}           
\begin{abstract}
Thanks to Hrushovski-Loeser's work on motivic Milnor fibers, we give a model-theoretic proof for the motivic Thom-Sebastiani theorem in the case of regular functions. Moreover, slightly extending of Hrushovski-Loeser's construction adjusted to Sebag, Loeser and Nicaise's motivic integration for formal schemes and rigid varieties, we formulate and prove an analogous result for formal functions. The latter is meaningful as it has been a crucial element of constructing Kontsevich-Soibelman's theory of motivic Donaldson-Thomas invariants.
\end{abstract}
\maketitle                 

\section{Introduction}
Let $f$ and $g$ be holomorphic functions on complex manifolds of dimensions $d_1$ and $d_2$, having isolated singularities at $\x$ and $\y$, respectively. Define $f\oplus g$ by $f\oplus g(x,y)=f(x)+g(y)$. Let $F_{f,\x}$ be the (topological) Milnor fiber of $(f,\x)$, the same for $(g,\y)$ and $(f \oplus g,(\x,\y))$. The original Thom-Sebastiani theorem \cite{ST} states that there exists an isomorphism between the cohomology groups 
$$H^{d_1+d_2-1}(F_{f\oplus g,(\x,\y)},\mathbb Q)\cong H^{d_1-1}(F_{f,\x},\mathbb Q)\otimes H^{d_2-1}(F_{g,\y},\mathbb Q)$$ 
compatible with the monodromies. Steenbrink in \cite{Steen} refined a conjecture on the Thom-Sebastiani theorem for the mixed Hodge structures, which was fulfilled later and independently by Varchenko \cite{V2} and Saito \cite{Sa}. In the letters to A'Campo (1972) and to Illusie (1999), Pierre Deligne discussed the $\ell$-adic version for an arbitrary field (rather than complex numbers), in which he replaced the Milnor fibers by the nearby cycles and used Laumon's construction of convolution product (cf. \cite[D\'efinition 2.7.2]{Lau}); this work recently has been fully realized by Fu \cite{Lei}. Furthermore, Denef-Loeser \cite{DL3} and Looijenga \cite{Loo} also provided proofs of the motivic version for motivic vanishing cycles in the case of fields of characteristic zero, from which the classical results were recovered without the hypothesis that $\x$ and $\y$ are isolated singularities.

We come back to the problem on the motivic Thom-Sebastiani theorem in the framework for the motivic Milnor fibers of {\it formal} functions. It has been likely a formally unsolved problem, but already used in Kontsevich-Soibelman's theory of motivic Donaldson-Thomas invariants for non-commutative Calabi-Yau threefolds (see \cite{KS}). Using Temkin's results on resolution of singularities of an excellent formal scheme \cite{Tem} and Denef-Loeser's formulas for the motivic Milnor fiber of a regular function \cite{DL1, DL4}, Kontsevich and Soibelman introduce in \cite{KS} the motivic Milnor fiber of a formal function. The motivic Thom-Sebastiani theorem for formal functions that concerns this notion is a key to construct the motivic Donaldson-Thomas invariants. In fact, it has the same interpretation as Denef-Loeser's and Looijenga's local version  (cf. \cite{DL3}, \cite{Loo}) and a complete proof for it should be required. This is the main purpose of the present article.

The motivic Milnor fiber of a regular function may be described in terms of resolution of singularity, after the works of Denef-Loeser \cite{DL1, DL4, DL5} and of Guibert-Loeser-Merle \cite{GLM1, GLM2, GLM3}. In particular, Guibert-Loeser-Merle had the refinement when applying this method to further extensions of the motivic Thom-Sebastiani theorem (see \cite{GLM1, GLM2, GLM3}). Recently, with the help of Hrushovski-Kazhdan's motivic integration, Hrushovski and Loeser \cite{HL} even give a more flexible manner to describe the motivic Milnor fiber in terms of the data of the corresponding analytic Milnor fiber (introduced by Nicaise-Sebag \cite{NS}). An important application of this approach is our proof of the integral identity conjecture in \cite{Thuong3}. Also in \cite{Thuong3}, a slight generalization of Hrushovski-Loeser's construction \cite{HL} combined with Nicaise's formula on volume Poincar\'e series \cite{Ni2} allows to interpret in the same way as in \cite{HL} the motivic Milnor fiber of a formal function. However, this method requires the restriction to studying over algebraically closed fields of characteristic zero (hence the hypothesis in the present work).

Our article is organized as follows. In Section \ref{sec2}, we recall some basic and essential backgrounds on the motivic Milnor fiber of a regular function, in which the local form of Denef-Loeser and Looijenga's motivic Thom-Sebastiani theorem is included (Theorem \ref{thom-sebastiani}), using the main references \cite{DL1, DL2, DL3, DL4, DL5} and \cite{Loo}. The local form states that 
$$\mathscr S_{f\oplus g,(\x,\y)}^{\phi}=\mathscr S_{f,\x}^{\phi}* \mathscr S_{g,\y}^{\phi},$$ 
where $\mathscr S_{f,\x}$ is the motivic Milnor fiber of $(f,\x)$, $\mathscr S_{f,\x}^{\phi}:=(-1)^{d_1-1}(\mathscr S_{f,\x}-1)$, the same for $(g,\y)$ and $(f\oplus g,(\x,\y))$, and $*$ is the convolution product (cf. Subsection \ref{TSconv}). Here, one does not need to assume that $\x$ and $\y$ are isolated singular points. Using the tools from \cite{HK} and \cite{HL}, recalled partly here in Section \ref{HLmorphismsssss}, we introduce a new proof for this formula in Section \ref{SSregular}. Notice that the previous formula lives in the monodromic Grothendieck ring $\mathscr M_k^{\hat{\mu}}$, by a technical reason, however, our proof only runs in a localization of $\mathscr M_k^{\hat{\mu}}$. 

In Section \ref{FormalTS}, we mark the highlights and the essences of motivic integration for special formal schemes, following \cite{Se}, \cite{LS}, \cite{NS}, \cite{Ni2} and \cite{Thuong3}. In particular, by \cite{Thuong3}, we show that Kontsevich-Soibelman's motivic Milnor fiber of a formal function and Nicaise's volume Poincar\'e series mention on the same thing and this can be also read off from the corresponding analytic Milnor fiber. Furthermore, we can use the model-theoretic tools recalled in Section \ref{HLmorphismsssss} to describe the volume Poincar\'e series, hence the motivic Milnor fiber of a formal function. The formal version of the motivic Thom-Sebastiani theorem has the same form as the regular one but $f$ and $g$ replaced by formal functions $\mathfrak f$ and $\mathfrak g$, respectively (Theorem \ref{TSformal}). It is proven in Section \ref{CC} using the development of tools in Section \ref{HLmorphismsssss} as well as some analogous techniques in the proof of the regular version in Section \ref{SSregular}. 

{\it Acknowledgement.} The author is grateful to Fran\c cois Loeser, Julien Sebag and Michel Raibaut for their useful discussions. He would like to thank the Centre Henri Lebesgue and the Universit\'e de Rennes 1 for awarding him a postdoctoral fellowship and an excellent atmosphere during his stays there. This work was partially supported par the Centre Henri Lebesgue in the program ``Investissements d'avenir'' --- ANR-11-LABX-0020-01. Earlier part of the manuscript was done at the Institut de Math\'ematiques de Jusieu, Universit\'e Pierre et Marie Curie, under Loeser's guidance, partially supported by ERC under the European Community's Seventh Framework Programme (FP7/2007-2013) / ERC Grant Agreement no. 246903/NMNAG.



\section{Preliminaries}\label{sec2}
Throughout the present article, we always assume that $ k$ is an algebraically closed field of characteristic zero.
\subsection{Grothendieck rings of algebraic varieties}
By definition, an {\it algebraic $k$-variety} is a separated reduced $k$-scheme of finite type. Let $\Var_k$ be the category of algebraic $k$-varieties, its morphisms are morphisms of algebraic $k$-varieties. The {\it Grothendieck group} $K_0(\Var_k)$ is an abelian group generated by symbols $[\mathscr X]$ for objects $\mathscr X$ in $\Var_k$ subject to the relations $[\mathscr X]=[\mathscr Y]$ if $\mathscr X$ and $\mathscr Y$ are isomorphic in $\Var_k$, $[\mathscr X]=[\mathscr Y]+[\mathscr X\setminus \mathscr Y]$ if $\mathscr Y$ is Zariski closed in $\mathscr X$. Moreover, $K_0(\Var_k)$ is also a ring with unit with respect to the cartesian product. Set $\Lbb:=[\mathbb A_k^1]$ and denote by $\mathscr M_k$ the localization of $K_0(\Var_k)$ with respect to the multiplicative system $\{\Lbb^i\mid i\in\mathbb N\}$. 

Let $\mu_m$ (or $\mu_m(k)$) be the group scheme of $m$th roots of unity in $k$. Varying $m\geq 1$ in $\mathbb N$, such schemes give rise to a projective system with respect to morphisms $\mu_{mn}\to\mu_{m}$ given by $\xi\mapsto \xi^n$, and its limit will be denoted by $\hat{\mu}$. A {\it good $\mu_m$-action} on an object $\mathscr X$ of $\Var_k$ is a group action of $\mu_m$ on $\mathscr X$ such that each orbit is contained in an affine $k$-subvariety of $\mathscr X$. A {\it good $\hat{\mu}$-action} on $\mathscr X$ is a $\hat{\mu}$-action which factors through a good $\mu_m$-action for some $m\geq 1$ in $\mathbb N$. 

The {\it $\hat{\mu}$-equivariant Grothendieck group} $K_0^{\hat{\mu}}(\Var_k)$ is an abelian group generated by the iso-equivariant classes of varieties $[\mathscr X,\sigma]$, with $\mathscr X$ an algebraic $k$-variety, $\sigma$ a good $\hat{\mu}$-action on $\mathscr X$, modulo the conditions $[\mathscr X,\sigma]=[\mathscr Y,\sigma|_{\mathscr Y}]+[\mathscr X\setminus \mathscr Y,\sigma|_{\mathscr X\setminus \mathscr Y}]$ if $\mathscr Y$ is Zariski closed in $\mathscr X$ and $[\mathscr X\times\mathbb A_k^n,\sigma]=[\mathscr X\times\mathbb A_k^n,\sigma']$ if $\sigma$, $\sigma'$ lift the same $\hat{\mu}$-action on $\mathscr X$ to an affine action on $\mathscr X\times\mathbb A_k^n$. In the present article we shall write $[\mathscr X,\sigma]$ simply by $[\mathscr X]$ when the $\hat{\mu}$-action $\sigma$ is clear. Similarly as previous, $K_0^{\hat{\mu}}(\Var_k)$ has a natural ring structure due to the cartesian product. Let $\mathscr M_k^{\hat{\mu}}$ denote $K_0^{\hat{\mu}}(\Var_k)[\Lbb^{-1}]$, it is the $\hat{\mu}$-equivariant version of $\mathscr M_k$ above. Let $\mathscr M_{k,\loc}^{\hat{\mu}}$ be the localization of $\mathscr M_k^{\hat{\mu}}$ with respect to the multiplicative family generated by the elements $1-\Lbb^i$, with $i\geq 1$ in $\mathbb N$. We shall also write $\loc$ for the localization morphism $\mathscr M_k^{\hat{\mu}}\to \mathscr M_{k,\loc}^{\hat{\mu}}$.

\subsection{Motivic Milnor fiber}
Let $\mathscr X$ be a pure $d$-dimensional smooth $ k$-variety, $f$ a non-constant regular function on $\mathscr X$, and $\x$ a closed point in the zero locus of $f$. Denote by $\mathscr{X}_{\x,m}$ (or $\mathscr{X}_{\x,m}(f)$) the set of arcs $\varphi(t)$ in $\mathscr X(k[t]/(t^{m+1}))$ originated at $\x$ with $f(\varphi(t))\equiv t^m\mod t^{m+1}$, which is a locally closed subvariety of $k$-variety $\mathscr X(k[t]/(t^{m+1}))$. Since $\mathscr{X}_{\x,m}$ is invariant by the $\hat{\mu}$-action on $\mathscr X(k[t]/(t^{m+1}))$ given by $\xi\cdot\varphi(t)=\varphi(\xi t)$, it defines an iso-equivariant class $[\mathscr{X}_{\x,m}]$ in $\mathscr M_k^{\hat{\mu}}$. The {\it motivic zeta function of $f$ at $\x$} is the formal series
$$Z_{f,\x}(T)=\sum_{m\geq 1}[\mathscr{X}_{\x,m}]\Lbb^{-md}T^m$$
with coefficients in $\mathscr M_k^{\hat{\mu}}$. By Denef-Loeser \cite{DL1}, $Z_{f,\x}(T)$ is a {\it rational} function, i.e., a $\mathscr M_k^{\hat{\mu}}$-linear combination of 1 and products finite (possibly empty) of $\Lbb^aT^b/(1-\Lbb^aT^b)$ with $(a,b)$ in $\mathbb Z\times\mathbb N_{>0}$. Remark that we can take by \cite{DL4} the limit $\lim_{T\to\infty}$ for rational functions such that $\lim_{T\rightarrow\infty}\left(\Lbb^aT^b/(1-\Lbb^aT^b)\right)=-1$. Then the {\it motivic Milnor fiber of $f$ at $x$} is defined as $-\lim_{T\to\infty}Z_{f,\x}(T)$ and denoted by $\mathscr S_{f,\x}$. This is a virtual variety in $\mathscr M_k^{\hat{\mu}}$.

\subsection{The motivic Thom-Sebastiani theorem for regular functions}\label{TSconv}
In this subsection, we restate the motivic Thom-Sebastiani theorem for motivic Milnor fibers.

Let us recall the concept of convolution product from \cite{DL3}, \cite{Loo} and \cite{GLM1}. Consider the Fermat varieties $F_0^m$ and $F_1^m$ in $\mathbb G_{m,k}^2$ defined by the equations $u^m+v^m=0$ and $u^m+v^m=1$, respectively. We endow with the standard $(\mu_m\times\mu_m)$-action on these varieties. If $\mathscr X$ and $\mathscr Y$ are algebraic $k$-varieties with $\mu_m$-action, one defines 
\begin{align*}
[\mathscr X]*[\mathscr Y]=-[F_1^m\times^{\mu_m\times\mu_m}(\mathscr X\times \mathscr Y)]+[F_0^m\times^{\mu_m\times\mu_m}(\mathscr X\times \mathscr Y)],
\end{align*}
where, for $i\in\{0,1\}$, 
$$F_i^m\times^{\mu_m\times\mu_m}(\mathscr X\times \mathscr Y)=F_i^m\times(\mathscr X \times \mathscr Y)/_{\sim}$$ 
with $(au,bv,x,y)\sim (u,v,ax,by)$ for any $a$, $b$ in $\mu_m$. The group scheme $\mu_m$ acts diagonally on $F_i^m\times^{\mu_m\times\mu_m}(\mathscr X\times \mathscr Y)$. Passing to the projective limit that $\mathscr M_k^{\hat{\mu}}$ equals $\underleftarrow{\lim} \mathscr M_k^{\mu_m}$, we get the convolution product $*$ on $\mathscr M_k^{\hat{\mu}}$. This product is commutative and associative (see for example \cite{GLM1}).

Let $f$ and $g$ be regular functions on smooth algebraic $k$-varieties $\mathscr X$ and $\mathscr Y$, respectively. Define $f\oplus g (x,y)=f(x)+g(y)$. For closed points $\x$ in $\mathscr X_0$ and $\y$ in  $\mathscr Y_0$, we set 
$$\mathscr S_{f,\x}^{\phi}=(-1)^{\dim\mathscr X-1}(\mathscr S_{f,\x}-1),\quad \mathscr S_{g,\y}^{\phi}=(-1)^{\dim\mathscr Y-1}(\mathscr S_{g,\y}-1).$$

\begin{theorem}[\cite{DL3}, \cite{Loo}]\label{thom-sebastiani}
The identity $\mathscr S_{f\oplus g,(\x,\y)}^{\phi}=\mathscr S_{f,\x}^{\phi}* \mathscr S_{g,\y}^{\phi}$ holds in $\mathscr M_{ k}^{\hat{\mu}}$.
\end{theorem}

\begin{remark}
In fact, in \cite{DL3} and \cite{Loo}, one proved the motivic Thom-Sebastiani theorem in the framework of motivic vanishing cycles, which implies Theorem \ref{thom-sebastiani}.
\end{remark}

\section{The motivic Thom-Sebastiani formula for formal functions}\label{FormalTS}
Let $\mathfrak X$ be a generically smooth special formal $k[[t]]$-scheme of relative dimension $d$, with reduction $\mathfrak X_0$ and structural morphism $\mathfrak f$. Let $\x$ be a closed point of $\mathfrak X_0$. 
\subsection{The motivic Milnor fiber of a formal function}
By \cite{Tem} (see also \cite{Ni2}), there exists a resolution of singularities $\mathfrak h: \mathfrak Y \to \mathfrak X$ of $\mathfrak X_0$. Let $\mathfrak E_i$, $i\in J$, be the irreducible components of $(\mathfrak Y_s)_{\mathrm{red}}$. Let $N_i$ be the multiplicity of $\mathfrak E_i$ in $\mathfrak Y_s$. We set $E_i=(\mathfrak E_i)_0$ for $i\in J$, $E_I=\bigcap_{i\in I}E_i$ and $E_I^{\circ}=E_I\setminus\bigcup_{j\not\in I}E_j$ for a nonempty subset $I$ of $J$. Let $\{U\}$ be a covering of $\mathfrak Y$ by affine open subschemes with $U\cap E_I^{\circ}\not=\emptyset$ such that, on this piece, $\mathfrak f\circ\mathfrak h= \widetilde{u}\prod_{i\in I}y_i^{N_i}$, where $\widetilde{u}$ is a unit, $y_i$ is a local coordinate defining $E_i$. Set $m_I:=\gcd(N_i)_{i\in I}$. One can construct as in \cite{DL5} an unramified Galois covering $\pi_I:\widetilde{E}_I^{\circ}\to E_I^{\circ}$ with Galois group $\mu_{m_I}$, which is given over $U\cap E_I^{\circ}$ by 
$$\{(z,y)\in \mathbb{A}_k^1\times(U\cap E_I^{\circ}) : z^{m_I}=\widetilde{u}(y)^{-1}\}.$$
$\widetilde{E}_I^{\circ}$ is endowed with a natural $\mu_{m_I}$-action good over $E_I^{\circ}$ obtained by multiplying the $z$-coordinate with elements of $\mu_{m_I}$. We also restrict this covering over $E_I^{\circ}\cap \mathfrak h^{-1}(\x)$ and obtain a class, written as $[\widetilde{E}_I^{\circ}\cap \mathfrak h^{-1}(\x)]$, in $\mathscr M_k^{\hat{\mu}}$. The {\it motivic Milnor fiber of the formal germ $(\mathfrak X,\x)$, or of $\mathfrak f$ at $\x$}, is defined to be the quantity 
\begin{align*}
\sum_{\emptyset\not=I\subset J} (1-\Lbb)^{|I|-1}[\widetilde E_I^{\circ}\cap \mathfrak h^{-1}(\x)]
\end{align*}
in $\mathscr M_k^{\hat{\mu}}$. We denote it by $\mathscr S(\mathfrak X,\x)$ or by $\mathscr S_{\mathfrak f,\x}$. By \cite[Lemma 5.7]{Thuong3}, using volume Poincar\'e series, $\mathscr S_{\mathfrak f,\x}$ is well defined, i.e., independent of the choice of the resolution of singularities $\mathfrak h$. 

\begin{remark}
Let $\widehat{\mathfrak X}_{\x}$ denote the formal completion of $\mathfrak X$ at $\x$, and let $\mathfrak f_{\x}$ be the structural morphism of $\widehat{\mathfrak X}_{\x}$, which is induced by $\mathfrak f$. We are able to use a resolution of singularity of $\mathfrak X$ at $\x$ to define the motivic Milnor fiber $\mathscr S_{\mathfrak f_{\x},\x}$. Then, it is clear that $\mathscr S_{\mathfrak f,\x}=\mathscr S_{\mathfrak f_{\x},\x}$.
\end{remark}

\subsection{Integral of a gauge form and volume Poincar\'e series}\label{subsec-LS}
\subsubsection{Stft formal schemes}
Assume that $\mathfrak X$ is a separated generically smooth formal $k[[t]]$-scheme topologically of finite type and that the relative dimension of $\mathfrak X$ is $d$. One may regard $\mathfrak X$ as the inductive limit of the $k[t]/(t^{m+1})$-schemes topologically of finite type $\mathscr X_m=(\mathfrak X, \mathcal O_{\mathfrak X}\otimes_{k[[t]]}k[t]/(t^{m+1}))$ in the category of formal $k[[t]]$-schemes. By Greenberg \cite{Gr}, there exists a unique $k$-scheme $\Gr_m(\mathscr X_m)$ topologically of finite type, up to isomorphism, which for any $k$-scheme $\mathscr Y$ admits a natural bijection 
$$\Hom_k(\mathscr Y,\Gr_m(\mathscr X_m))\to\Hom_{\Spec(k)}(\mathscr Y\times_kk[t]/(t^{m+1}),\mathscr X_m).$$
These $k$-schemes $\Gr_m(\mathscr X_m)$ together with the natural translation gives rise to a projective system, we denote its limit by $\Gr(\mathfrak X)$ (cf. \cite{Se}, \cite{LS}). We denote by $\pi_m$ the canonical projection $\Gr(\mathfrak X)\to\Gr_m(\mathscr X_m)$. See more in \cite{Gr} for some basic properties of the functor $\Gr$.

By \cite{Se}, \cite{LS}, the motivic measure of a stable cylinder $A$ in $\Gr(\mathfrak X)$ is the following 
\begin{align}\label{measure}
\mu(A)=[\pi_{\ell}(A)]\Lbb^{-(\ell+1)d}
\end{align}
for $\ell\in\mathbb N$ large enough. Let $\alpha: A\to \mathbb Z \cup \{\infty\}$ be a function on $A$ that takes only a finite number of values such that every fiber $\alpha^{-1}(m)$ is a stable cylinder in $\Gr(\mathfrak X)$. Let $\omega$ be a gauge form on $\mathfrak X_{\eta}$. By \cite[Proposition 1.5]{BLR} (see also \cite{LS}), there exists a canonical isomorphism $\Omega_{\mathfrak X_{\eta}}^d(\mathfrak X_{\eta})\cong \Omega_{\mathfrak X|k[[t]]}^d(\mathfrak X)\otimes_{k[[t]]}k((t))$, thus there exist an $n$ in $\mathbb N$ and a differential form $\widetilde\omega$ in $\Omega_{\mathfrak X|k[[t]]}^d(\mathfrak X)$ such that $\omega=t^{-n}\widetilde\omega$. Let $\varphi$ be a point of $\Gr(\mathfrak X)$ outside $\Gr(\mathfrak X_{\sing})$. Then, we can regard it as a morphism of formal schemes $\Spf(k[[t]])\to \mathfrak X$, or as a morphism of rings $\mathcal O_{\mathfrak X}(\mathfrak X)\to k[[t]]$. Thus it induces a morphism of rings $\widetilde\varphi:\varphi^*\Omega_{\mathfrak X|k[[t]]}^d(\mathfrak X)\to k[[t]]$, which is a surjection. One defines 
\begin{align}\label{order}
\ord(\widetilde\omega)(\varphi)=\ord_t(\widetilde\varphi(\varphi^*\widetilde\omega))\ \ \text{and}\ \ \ord_{\mathfrak X}(\omega)=\ord(\widetilde\omega)-n.
\end{align}
The latter is independent of the choice of $\widetilde\omega$ (cf. \cite{LS}). Since $\omega$ is a gauge form, it follows from \cite[Proof of 4.1.2]{LS} that $\ord_{\mathfrak X}(\omega)$ is a integer-valued function taking only a finite number of values and that its fibers are stable cylinder. Then one defines (cf. \cite{Se}, \cite{LS})
\begin{align}\label{integral}
\int_{\mathfrak X_{\eta}}|\omega|:=\sum_{m\in\mathbb Z}\mu\left(\{\varphi\in\Gr(\mathfrak X)\mid \ord_{\mathfrak X}(\omega)(\varphi)=m\}\right)\Lbb^{-m}\in \mathscr M_k.
\end{align}

\subsubsection{Special formal schemes}\label{3.2.2}
We consider the more general case where $\mathfrak X$ is a generically smooth special formal $k[[t]]$-schemes (see \cite{Ab} for definition). Let $\mathfrak Y \to\mathfrak X$ be a {\it N\'eron smoothening for $\mathfrak X$}, i.e. a morphism of special formal $k[[t]]$-schemes, $\mathfrak Y$ adic smooth over $k[[t]]$, inducing an open embedding $\mathfrak Y_{\eta} \to \mathfrak X_{\eta}$ with $\mathfrak Y_{\eta}\widehat{\otimes}_{k((t))}K=\mathfrak X_{\eta}\widehat{\otimes}_{k((t))}K$ for any finite unramified extension $K$ of $k((t))$. It exists by \cite{Ni2}, furthermore, we are able to (and we shall from now on) choose $\mathfrak Y$ to be separated generically smooth formal $k[[t]]$-scheme topologically of finite type. Using \cite[Propositions 4.7, 4.8]{Ni2}, for any gauge form $\omega$ on $\mathfrak X_{\eta}$, we define 
$$\int_{\mathfrak X_{\eta}}|\omega|:= \int_{\mathfrak Y_{\eta}}|\omega| \in \mathscr M_k.$$

For any $m$ in $\mathbb N_{>0}$, let $\mathfrak X(m):=\mathfrak X\widehat{\otimes}_{k[[t]]}k[[t^{1/m}]]$, $\mathfrak X_{\eta}(m):=\mathfrak X_{\eta}\widehat{\otimes}_{k((t))}k((t^{1/m}))$ and $\omega(m)$ the pullback of $\omega$ via the natural morphism $\mathfrak X_{\eta}(m)\to \mathfrak X_{\eta}$. The N\'eron smoothening $\mathfrak Y \to\mathfrak X$ for $\mathfrak X$ induces a N\'eron smoothening $\mathfrak Y(m)\to\mathfrak X(m)$ for $\mathfrak X(m)$ and $\mathfrak Y(m)$ is also topologically of finite type, like $\mathfrak Y$. The canonical $\mu$-action on $\Gr(\mathfrak Y(m)$ is given by $a\varphi(t^{1/m})=\varphi(at^{1/m})$. It induces a $\mu_m$-action on $\int_{\mathfrak X_{\eta}}|\omega|$, thus we regard $\int_{\mathfrak X_{\eta}}|\omega|$ as an element of $\mathscr M_k^{\hat{\mu}}$.

\subsubsection{Volume Poincar\'e series}
Let $\mathfrak X$ be a generically smooth special formal $k[[t]]$-schemes, $\x$ a closed point of $\mathfrak X_0$ and $\widehat{\mathfrak X}_{\x}$ the formal completion of $\mathfrak X$ at $\x$. Denoting by $]\x[$ the tube of $\x$, namely {\it the analytic Milnor fiber of $\mathfrak f$ at $\x$} (cf. \cite{NS}), we have the canonical isomorphism $]\x[\ \cong (\widehat{\mathfrak X}_{\x})_{\eta}$. Set $]\x[_m:=]\x[\times_{k((t))}k((t^{1/m}))$. Let us consider the volume Poincar\'e series of $(]\x[,\omega)$, where $\omega$ is a gauge form on $]\x[$, (cf. \cite{Ni2})
\begin{align*}
S(]\x[,\omega; T):=\sum_{m\geq 1}\left(\int_{]\x[_m}|\omega(m)|\right)T^m \in \mathscr M_k^{\hat{\mu}}[[T]].
\end{align*}

\begin{remark}
More generally, the volume Poincar\'e series of separated generically smooth formal schemes topologically of finite type (resp. separated quasi-compact smooth rigid varieties) were introduced and studied first by Nicaise-Sebag in \cite{NS}. After that, Nicaise \cite{Ni2} studied these objects in the framework of generically smooth special formal schemes (resp. bounded smooth rigid varieties).

In practice, one may assume $\omega$ is $\widehat{\mathfrak X}_{\x}$-bounded, i.e., $\omega$ lies in the image of the natural map $(\Omega_{\widehat{\mathfrak X}_{\x}|k[[t]]}^d\otimes_{k[[t]]}k((t)))(\widehat{\mathfrak X}_{\x})\to \Omega_{]\x[|k((t))}^d(]\x[)$ (cf. \cite[Definition 2.11]{Ni2}).  
\end{remark}

Since $k$ is an algebraically closed field, $S(]\x[,\omega; T)$ is independent of the choice of the uniformizing parameter $t$. Indeed, let $t'$ be another uniformizing parameter for $k[[t]]$. Then $t'=\alpha t$, where $\alpha=\alpha(t)\in k[[t]]$ and $\alpha(0)\in k^{\times}$. Since $k$ contains all roots, the $m$th roots of $\alpha$ are again in $k[[t]]$. This induces a canonical isomorphism of $k((t))$-fields $k((t^{1/m}))\to k((t^{\prime 1/m}))$, that implies the previous claim. By Nicaise \cite[Corollary 7.13]{Ni2}, if the gauge form $\omega$ is $\widehat{\mathfrak X}_{\x}$-bounded, this series $S(]\x[,\omega; T)$ is a rational function. 

\begin{proposition}
With the notation and the hypotheses as previous, the following identity holds in $\mathscr M_k^{\hat{\mu}}$: 
$$\mathscr S_{\mathfrak f,\x}=-\Lbb^d\lim_{T\to\infty}\sum_{m\geq 1}\left(\int_{]\x[_m}|\omega(m)|\right)T^m.$$
\end{proposition}

\begin{proof}
The identity is true in $\mathscr M_k$ because of the definition of $\mathscr S_{\mathfrak f,\x}$ as well as Nicaise's formula for $-\lim_{T\to\infty}S(]\x[,\omega; T)$ in \cite[Proposition 7.36]{Ni2}. To see that it is true in $\mathscr M_k^{\hat{\mu}}$, we refer to the proof of Lemma 5.7 in \cite{Thuong3}.
\end{proof}


\subsection{Statement of result for formal functions}\label{statement}
Given integers $d_1\geq 1$ and $d_2\geq 1$. Let $f$ be a formal power series in $k[[x]]$ with $f(0)=0$ and $g$ in $k[[y]]$ with $g(0)=0$. Here $x=(x_1,\dots,x_{d_1})$, $y=(y_1,\dots,y_{d_2})$ and we write by the same symbol $0$ for the origin of $\mathbb A_k^{d_1}$, $\mathbb A_k^{d_2}$ or $\mathbb A_k^1$ (whenever necessary, e.g., in Section \ref{CC}, however, we shall write $0_{d_i}$ for the origin of $\mathbb A_k^{d_i}$, $i\in\{1, 2\}$). Let us consider the following special formal $k[[t]]$-schemes
\begin{align*}
\mathfrak X:&=\Spf\left(k[[t,x]]/(f(x)-t)\right),\\ 
\mathfrak Y:&=\Spf\left(k[[t,y]]/(g(y)-t)\right),\\
\mathfrak X\oplus \mathfrak Y:&=\Spf\left(k[[t,x,y]]/(f(x)+g(y)-t)\right),
\end{align*}
with structural morphisms $\mathfrak f$, $\mathfrak g$ and $\mathfrak f\oplus \mathfrak g$ induced by $f$, $g$ and $f\oplus g$, respectively. 
Set $\mathscr S_{\mathfrak f,0}^{\phi}:=(-1)^{d_1-1}(\mathscr S_{\mathfrak f,0}-1)$ and the same for $\mathfrak g$ and $\mathfrak f\oplus \mathfrak g$. 
We now set up the statement of the motivic Thom-Sebastiani theorem for formal schemes and then prove it in the setting of $\mathscr M_{k,\loc}^{\hat{\mu}}$, using Hrushovski-Kahdan's integration \cite{HK} via the work of Hrushovski and Loeser \cite{HL}. 

\begin{theorem}\label{TSformal}
The identity $\mathscr S_{\mathfrak f\oplus \mathfrak g,(0,0)}^{\phi}=\mathscr S_{\mathfrak f,0}^{\phi}*\mathscr S_{\mathfrak g,0}^{\phi}$ holds in $\mathscr M_{k,\loc}^{\hat{\mu}}$.
\end{theorem}

The complete proof is given in Section \ref{CC}.
%
\section{Extension of Hrushovski-Loeser's morphism}\label{HLmorphismsssss}

\subsection{The theory $\ACVF_{k((t))}(0,0)$}
We consider the theory $\ACVF_{k((t))}(0,0)$ of algebraically closed valued fields of equal characteristic zero that extend $k((t))$ (cf. \cite{HK}). Its sort $\VF$ admits the language of rings, while the sort $\RV$ is endowed with abelian group operations $\cdot$, $/$, a unary predicate $\k^{\times}$ for a subgroup, a binary operation $+$ on $\k=\k^{\times}\cup\{0\}$. We also have an imaginary sort $\Gamma$ that is with a uniquely divisible abelian group. For a model $L$ of this theory, let $R_L$ (resp. $\mathfrak m_L$) denote its valuation ring (resp. the maximal ideal of $R_L$). The following are the ``elementary" $L$-definable sets of $\ACVF_{k((t))}(0,0)$:
$$\VF(L)=L,\ \RV(L)=L^{\times}/(1+\mathfrak m_L),\ \Gamma(L)=L^{\times}/R_L^{\times},\ \k(L)=R_L/\mathfrak m_L.$$
In general, a {\it definable subset} of $\VF^n(L)$ is a finite Boolean combination of set of the forms $\val(f_1)\leq \val(f_2)$ or $f_3=0$, where $f_i$ are polynomials with coefficients in $k((t))$. The same definition may apply to {\it definable subsets} of $\RV^n(L)$, $\Gamma^n(L)$ or $\k^n(L)$. Correspondingly, there are natural maps between these sets $\rv: \VF\to\RV$, $\val: \VF\to \Gamma$, $\val_{\rv}:\RV\to\Gamma$ and $\res:R_L\to\k(L)$. There is an exact sequence of groups 
$$1\to \k^{\times}\to \RV\stackrel{\val_{\rv}}{\to}\Gamma\to 0.$$


\subsection{Measured categories {\rm (following \cite{HK})}}\label{subsec4.3}
\subsubsection{$\VF$-categories} 
Let $\mu_{\Gamma}\VF$ be the category of $k((t))$-definable sets (or definable sets, for short) endowed with definable volume forms, up to $\Gamma$-equivalence. One may show that it is graded via the following subcategories $\mu_{\Gamma}\VF[n]$, $n\in \mathbb N$. An object of $\mu_{\Gamma}\VF[n]$ is a triple $(X,f,\varepsilon)$ with $X$ a definable subset of $\VF^{\ell}\times \RV^{\ell'}$, for some $\ell$, $\ell'$ in $\mathbb N$, $f:X\to \VF^n$ a definable map with finite fibers and $\varepsilon:X\to \Gamma$ a definable function; a morphism from $(X,f,\varepsilon)$ to $(X',f',\varepsilon')$ is a definable {\it essential} bijection $F:X\to X'$ such that 
$$\varepsilon=\varepsilon'\circ F+\val(\Jac F)$$ 
away from a proper closed subvariety of $X$. Here, that $F:X\to X'$ is an essential bijection means that there exists a proper closed subvariety $Y$ of $X$ such that $F|_{X\setminus Y}: X\setminus Y\to X'\setminus F(Y)$ is a bijection (see \cite[Subsection 3.8]{HK}).

Let $\mu_{\Gamma}\VF^{\bdd}[n]$ be the full subcategory of $\mu_{\Gamma}\VF[n]$ whose objects are bounded definable sets with bounded definable forms $\varepsilon$. If considering $\varepsilon:X\to \Gamma$ as the zero function, we obtain the categories $\vol\VF$ and $\vol\VF[n]$ as well as $\vol\VF^{\bdd}$ and $\vol\VF^{\bdd}[n]$. In this case, the measure preserving property of a morphism $F$ is characterized by the condition $\val(\Jac F)=0$, outside a proper closed subvariety.

\begin{convention}
For simplicity, we shall omit the symbol $f$ in the triple $(X,f,\varepsilon)$ if no possibility of confusion appears.
\end{convention}

\subsubsection{$\RV$-categories}
Similarly, we consider the category $\mu_{\Gamma}\RV$ graded by $\mu_{\Gamma}\RV[n]$, $n\in \mathbb N$. By definition, an object of $\mu_{\Gamma}\RV[n]$ is a triple $(X,f,\varepsilon)$ with $X$ a definable subset of $\RV^{\ell}$, for some $\ell$ in $\mathbb N$, $f:X\to\RV^n$ a definable map with finite fibers, and $\varepsilon: X\to \Gamma$ a definable function; a morphism $(X,f,\varepsilon)\to (X',f',\varepsilon')$ is a definable bijection $F:X\to X'$ such that 
$$\varepsilon+\sum_{i=1}^n\val_{\rv}(f_i)= \varepsilon'\circ F+\sum_{i=1}^n\val_{\rv}(f'_i\circ F)$$ 
away from a proper closed subvariety (the {\it measure preserving} property). The category $\mu_{\Gamma}\RES[n]$ is defined as the full subcategory of $\mu_{\Gamma}\RV[n]$ such that, for each object $(X,f,\varepsilon)$, $\val_{\rv}(X)$ is a finite set. The category $\mu_{\Gamma}\RV^{\bdd}$ is defined as $\mu_{\Gamma}\RV$ with $\val_{\rv}$-image of objects bounded below. In the case where, for each object $(X,f,\varepsilon)$ of one of the previous categories, taking $\varepsilon$ being the zero function, we get the subcategories $\vol\RV$, $\vol\RV^{\bdd}$ and $\vol\RES$.

In the present article, we also consider $\RES$, a category defined exactly as $\vol\RES$ but the measure preserving property is not required for morphisms.

\subsubsection{$\Gamma$-categories}
The category $\mu\Gamma[n]$ consists of pairs $(\Delta,l)$ with $\Delta$ a definable subset of $\Gamma^n$ and $l:\Delta\to\Gamma$ a definable map. A morphism $(\Delta,l)\to (\Delta',l')$ is a definable bijection $\lambda:\Delta\to\Delta'$ which is liftable to a definable bijection $\val_{\rv}^{-1}\Delta\to \val_{\rv}^{-1}\Delta'$ such that 
$$|x|+l(x)=|\lambda(x)|+l'(\lambda(x)).$$ 
The category $\mu\Gamma^{\bdd}[n]$ is the full subcategory of $\mu\Gamma[n]$ such that, for each object $(\Delta,l)$ of $\mu\Gamma^{\bdd}[n]$, there exists a $\gamma\in\Gamma$ with $\Delta\subset [\gamma,\infty)^n$. By definition, the categories $\mu\Gamma$ and $\mu\Gamma^{\bdd}$ are the direct sums $\bigoplus_{n\geq 1}\mu\Gamma[n]$ and $\bigoplus_{n\geq 1}\mu\Gamma^{\bdd}[n]$, respectively. The subcategories whose objects are of the form $(\Delta,0)$ will be denoted by $\vol\Gamma$ and $\vol\Gamma^{\bdd}$.


\subsection{Structure of $K(\mu_{\Gamma}\VF^{\bdd})$}
Let $\mathscr C$ be one of the categories in Subsection \ref{subsec4.3}. Then, as in \cite{HK}, we denote the Grothendieck semiring of $\mathscr C$ by $K_+(\mathscr C)$ and the associated ring by $K(\mathscr C)$. By \cite{HK}, there is a natural morphism of 
\begin{align}\label{morphismN}
N:K_+(\mu\Gamma^{\bdd})\otimes K_+(\mu_{\Gamma}\RES)\to K_+(\mu_{\Gamma}\VF^{\bdd})
\end{align}
constructed as follows. Note that two objects admitting a morphism $\lambda$ in $\mu\Gamma^{\bdd}[n]$ define the same element in $K_+(\mu\Gamma^{\bdd}[n])$, hence $\lambda$ lifts to a morphism in $\mu_{\Gamma}\VF^{\bdd}[n]$ between their pullbacks. Thus there exists a natural morphism $K_+(\mu\Gamma^{\bdd}[n])\to K_+(\mu_{\Gamma}\VF^{\bdd})$ mapping the class of $(\Delta,l)$ to the class of $(\val^{-1}(\Delta),l\circ\val)$. Also, for each object $(X,f,\varepsilon)$ in $\mu_{\Gamma}\RES[n]$, we may consider an \'etale map $\ell:X\to\k^n$. By this, we have the natural morphism $K_+(\mu_{\Gamma}\RES[n])\to K_+(\mu_{\Gamma}\VF^{\bdd})$ by sending the class of $(X,f,\varepsilon)$ to the class of $(X\times_{\ell,\res}R^n,\pr_1\circ \varepsilon)$. In particular, if $X$ is Zariski open in $\k^n$, then $X\times_{\ell,\res}R^n$ is simply $\res^{-1}(X)$.

\begin{theorem}[Hrushovski-Kazhdan \cite{HK}]
The morphism $N$ is a surjection. Moreover, it also induces a surjective morphism $N$ between the associated rings.
\end{theorem}

The description of $N$, or more precisely, $N^{-1}$ modulo $\ker(N)$, in \cite{HK} and \cite{HL}, is slightly more explicit and more intrinsic. Indeed, one first constructs the natural morphism 
$$K_+(\mu\Gamma^{\bdd})\otimes K_+(\mu_{\Gamma}\RES)\to K_+(\mu_{\Gamma}\RV^{\bdd})$$ 
due to the inclusion $\RES\subset\RV$ and the valuation map $\val_{\rv}$ (cf. \cite{HK} or \cite{HL}). This morphism is a surjection, its kernel is generated by $1\otimes[\val_{\rv}^{-1}(\gamma)]_1-[\gamma]_1\otimes 1$, with $\gamma$ definable in $\Gamma$. The subscript 1 means that the classes are in degree 1. Secondly, the canonical morphism 
$$K_+(\mu_{\Gamma}\VF^{\bdd}[n])\to K_+(\mu_{\Gamma}\RV^{\bdd}[n])/_{[1]_1\sim[\RV^{>0}]_1}$$ 
induces by the map $\Ob\mu_{\Gamma}\RV[n]\to\Ob\mu_{\Gamma}\VF[n]$ sending $(X,f,\varepsilon)$ to $(\mathbf L X,\mathbf L f,\mathbf L\varepsilon)$, 
where $\mathbf L X=X\times_{f,\rv}(\VF^{\times})^n$, $\mathbf L f(a,b)=f(a,\rv(b))$ and $\mathbf L\varepsilon (a,b)=\varepsilon(a,\rv(b))$.

\begin{remark}\label{rk1}
According to \cite[Proposition 10.10]{HK}, an element of $K_+(\mu_{\Gamma}\RV^{\bdd})$ may be written as a finite sum of elements of the form $[(X\times\val_{\rv}^{-1}(\Delta),f,\varepsilon)]$. Furthermore, an argument in the proof of \cite[Proposition 10.10]{HK} implies a fact that $[(X\times\val_{\rv}^{-1}(\Delta),f,\varepsilon)]=[(X,f_0,1)]\otimes [(\Delta,l)]$, where $f_0:X\to \RV^n$, $l:\Delta\to\Gamma$ are some definable functions.
\end{remark}


\subsection{Extending Hrushovski-Loeser's construction}\label{HLconst} 
\subsubsection{The morphisms $h_m$ and $\widetilde h_m$}
From now on, we shall denote by $!K(\RES)$ the quotient of $K(\RES)$ subject to the relations $[\val_{\rv}^{-1}(a)]=[\val_{\rv}^{-1}(0)]$ for $a$ in $\Gamma$, and by $!K(\RES)[\Lbb^{-1}]_{\loc}$ the localization of $!K(\RES)[\Lbb^{-1}]$ with respect to the multiplicative family generated by $1-[\mathbb A^1]^i$, $i\geq 1$.  Let $m, n$ be in $\mathbb N$, $m\geq 1$, $(\Delta,l)$ in $\mu\Gamma^{\bdd}[n]$ and $e$ in $\Gamma$ with $me\in\mathbb Z$. Set $\Delta(m)=\Delta\cap(1/m \mathbb Z)^n$, $\Delta_{l,e}=l^{-1}(e)$ and
\begin{align*}
\alpha_m(\Delta,l)&=\sum_{e\in \Gamma, me\in\mathbb Z}\sum_{\gamma\in \Delta_{l,e}(m)}\Lbb^{-m(|\gamma|+e)}(\Lbb-1)^n\\
&=\sum_{e\in\mathbb Z}\sum_{\gamma\in \Delta_{l,e/m}(m)}\Lbb^{-m|\gamma|-e}(\Lbb-1)^n.
\end{align*}
It is clear that $\alpha_m(\Delta,l)$ is an element of $!K(\RES)[\Lbb^{-1}]_{\loc}$, and moreover, $\alpha_m$ is independent of the choice of coordinates for $\Gamma^n$. Indeed, let $\lambda$ be the morphism in $\mu\Gamma^{\bdd}$ from $(\Delta_{l,e},l|_{\Delta_{l,e}})$ to $(\Delta',l')$. Then $|\lambda(\gamma)|+l'(\lambda(\gamma))=|\gamma|+l(\gamma)=|\gamma|+e$ and the claim follows. Thus $\alpha_m$ defines a natural morphism of rings 
$$\alpha_m: K(\mu\Gamma^{\bdd})\to !K(\RES)[\Lbb^{-1}]_{\loc}.$$
By using \cite{HL}, for any $\widetilde\Delta$ in $\vol\Gamma^{\bdd}$, one sets $\widetilde\alpha_m(\widetilde\Delta)=\sum_{\gamma\in \widetilde\Delta(m)}\Lbb^{-m|\gamma|}(\Lbb-1)^n$ and obtains a morphism of rings 
$$\widetilde\alpha_m: K(\vol\Gamma^{\bdd})\to !K(\RES)[\Lbb^{-1}]_{\loc}.$$
Thus we can consider $\alpha_m$ as an extension of $\widetilde\alpha_m$; moreover, 
\begin{align}\label{deco}
\alpha_m(\Delta,l)=\sum_{e\in\mathbb Z}\widetilde\alpha_m(\Delta_{l,e/m})\Lbb^{-e}.
\end{align}

We are able to construct a morphism $\beta_m:K(\mu_{\Gamma}\RES)\to !K(\RES)[\Lbb^{-1}]_{\loc}$ by using Hrushovski-Loeser's method. Thanks to Remark \ref{rk1}, however, it suffices to define value under $\beta_m$ of elements of the form $[(X,f,1)]$ with $(X,f,1)$ an object in $\mu_{\Gamma}\RES$. Assume that $f(X)\subset V_{\gamma_1}\times\cdots\times V_{\gamma_n}$, i.e., $\val_{\rv}(f_i(x))=\gamma_i$ for every $x$ in $X$. We set $\beta_m(X,f,1)=[X](\Lbb^{-1}[1]_1)^{m|\gamma|}$ if $m\gamma\in\mathbb Z^n$ and $\beta_m(X,f,1)=0$ otherwise.

There are two steps to check that $\ker(\widetilde\alpha_m\otimes \beta_m)$ is contained in $\ker(N_0)$, where $N_0$ is $N$ reduced to the volume version (for the structure of $K(\vol\VF^{\bdd})$). These steps correspond to the factorization of $N_0$ into $K(\vol\Gamma^{\bdd})\otimes K(\vol\RES)\to K(\vol\RV^{\bdd})$ and $K(\vol\VF^{\bdd}[n])\to K(\vol\RV^{\bdd}[n])/_{[1]_1\sim[\RV^{>0}]_1}$. Hrushovski and Loeser \cite{HL} passed these by direct computation. This can be applied to show that $\ker(\alpha_m\otimes \beta_m)$ is contained in $\ker(N)$. Consequently, we obtain from the tensor products $\widetilde\alpha_m\otimes \beta_m$ and $\alpha_m\otimes \beta_m$ morphisms of rings 
$$\widetilde h_m:K(\vol\VF^{\bdd})\to !K(\RES)[\Lbb^{-1}]_{\loc}$$ 
and 
$$h_m:K(\mu_{\Gamma}\VF^{\bdd})\to !K(\RES)[\Lbb^{-1}]_{\loc}.$$
Moreover, there is a presentation of $h_m$ in terms of $\widetilde h_m$ induced from (\ref{deco}). Namely, we have the following lemma whose proof is trivial and left to the reader.

\begin{lemma}\label{simpleee}
$h_m([(X,\varepsilon)])=\sum_{e\in\mathbb Z}\widetilde h_m([\varepsilon^{-1}(e/m)])\Lbb^{-e}$ (in $!K(\RES)[\Lbb^{-1}]_{\loc}$).
\end{lemma}

\subsubsection{The morphism $h$}\label{4.4.2}
We also use the morphisms in \cite[Subsection 8.5]{HL} with their restriction, namely, $\alpha: K(\vol\Gamma^{\bdd})\to!K(\RES)[\Lbb^{-1}]$ and $\beta: K(\vol\RES)\to!K(\RES)[\Lbb^{-1}]$. By definition, $\beta([X])=[X]$, $\alpha([\Delta])=\chi(\Delta)(\Lbb-1)^n$ if $\Delta$ is a definable subset of $\Gamma^n$, where $\chi$ is the o-minimal Euler characteristic in the sense of \cite[Lemma 9.5]{HK}. Since $\ker(\alpha\otimes\beta)$ is contained in $\ker(N_0)$, it gives rise to a morphism of rings 
$$K(\vol\VF^{\bdd})\to!K(\RES)[\Lbb^{-1}].$$
The composition of it with the localization $!K(\RES)[\Lbb^{-1}]\to !K(\RES)[\Lbb^{-1}]_{\loc}$ will be denoted by $h$.

\begin{proposition}\label{HLmorph}
The formal series $Z'(X,\varepsilon)(T):=\sum_{m\geq 1}h_m([(X,\varepsilon)])T^m$ is a rational function. Moreover, we have $\lim_{T\to\infty}Z'(X,\varepsilon)(T)=-h([X])$.
\end{proposition}

\begin{proof}
It is similar to the proof of \cite[Proposition 8.5.1]{HL}.
\end{proof}


\subsection{Endowing with a $\hat{\mu}$-action and the morphisms $\hl_m$, $\widetilde\hl_m$ and $\hl$}\label{action}
First, let us recall \cite[4.3]{HL}). Define a series $\left\{t_m\right\}_{m\geq 1}$ by setting $t_1=t$, $t_{nm}^m=t_n$, $n\geq 1$. For a $\k((t))$-definable set $X$ over $\RES$, we may assume $X\subset V_{i_1/m}\times\cdots\times V_{i_n/m}$ for some $n$, $m$ and $i_j$'s. It is endowed with a natural action $\delta$ of $\mu_m$. Now the $\k((t^{1/m}))$-definable function 
$$(x_1,\dots,x_n)\mapsto (x_1/\rv(t_m^{i_1}),\dots, x_n/\rv(t_m^{i_n}))$$ 
maps $X$ to a constructible subset $Y$ of $\mathbb A_{\k}^n$, where $Y$ is endowed with a $\mu_m$-action induced from $\delta$. The correspondence $X\mapsto Y$ in its turn defines a morphism of rings $!K(\RES)[\Lbb^{-1}]\to !K_0^{\hat{\mu}}(\Var_k)[\Lbb^{-1}]$ (\cite[Lemma 10.7]{HK}, \cite[Proposition 4.3.1]{HL}). Here, by definition, $!K_0^{\hat{\mu}}(\Var_k)$ is the quotient of $K_0^{\hat{\mu}}(\Var_k)$ by identifying all the classes $[\mathbb G_m, \sigma]$ with $\sigma$ a $\hat{\mu}$-action on $\mathbb G_m$ induced by multiplication by roots of $1$. The previous morphism together with the natural one $!K^{\hat{\mu}}(\Var_k)[\Lbb^{-1}]\to \mathscr M_k^{\hat{\mu}}$ induces the following morphisms of rings, both are denoted by $\Theta$,
$$!K(\RES)[\Lbb^{-1}] \to \mathscr M_k^{\hat{\mu}}\ \text{and}\ !K(\RES)[\Lbb^{-1}]_{\loc} \to \mathscr M_{k,\loc}^{\hat{\mu}}.$$

We now define ring morphisms $\hl_m:=\Theta\circ h_m$, $\widetilde\hl_m:=\Theta\circ \widetilde h_m$ and $\hl:=\Theta\circ h$ with the same target $\mathscr M_{k,\loc}^{\hat{\mu}}$. In fact, while $\hl_m$ has the source $K(\mu_{\Gamma}\VF^{\bdd})$, $\widetilde \hl_m$ and $\hl$ starts from $K(\vol\VF^{\bdd})$. Similar to Lemma \ref{simpleee} and Proposition \ref{HLmorph}, we get 

\begin{lemma}\label{simp}
$\hl_m([(X,\varepsilon)])=\sum_{e\in\mathbb Z}\widetilde \hl_m([\varepsilon^{-1}(e/m)])\Lbb^{-e}$ (in $\mathscr M_{k,\loc}^{\hat{\mu}}$).
\end{lemma}

\begin{proposition}\label{8May}
The formal series $Z(X,\varepsilon)(T):=\sum_{m\geq 1}\hl_m([(X,\varepsilon)])T^m$ is a rational function. Moreover, we have $\lim_{T\to\infty}Z(X,\varepsilon)(T)=-\hl([X])$.
\end{proposition}


\subsection{Description of the motivic Milnor fibers}\label{subsec-descriptions} 
\subsubsection{Regular case}
Let $\gamma$ be in $\Gamma$. A definable subset $X$ of $\VF^{\ell}\times\RV^{\ell'}$ is {\it $\gamma$-invariant} if, for any $(x,x')\in\VF^{\ell}\times\RV^{\ell'}$ and any $(y,y')\in\VF^{\ell}\times\RV^{\ell'}$ with $\val(y)\geq\gamma$, both $(x,x')$ and $(x,x')+(y,y')$ simultaneously belong to either $X$ or the complement of $X$ in $\VF^{\ell}\times\RV^{\ell'}$. By \cite[Lemma 3.1.1]{HL}, any bounded definable subset of $\VF^{\ell}$ that is closed in the valuation topology is $\gamma$-invariant for some $\gamma$ in $\Gamma$. 

Assume that $X$ is a $\gamma$-invariant definable subset of $\VF^n\times \RV^{\ell'}$, where $\gamma$ is in $(1/m)\mathbb Z\subset \Gamma$. By \cite{HK}, $X(k((t^{1/m})))$ are the pullback of some definable subset $X[m;\gamma]$ of $\left( k[t^{1/m}]/t^{\gamma}\right)^n\times\RV^{\ell'}$ and the projection $X[m;\gamma]\to\VF^n$ is a finite-to-one map. If $\gamma'$ is in $\Gamma$ with $\gamma'\geq\gamma$, the equality $[X[m;\gamma']]=[X[m;\gamma]]\Lbb^{nm(\gamma'-\gamma)}$ holds in $!K(\vol\RES[n])$, thus $[X[m;\gamma]]\L^{-nm\gamma}$ in $!K(\RES)[\L^{-1}]$ is independent of the choice of $\gamma$ large enough. For brevity, we shall write $\widetilde{X}[m]$ for the quantity $[X[m;\gamma]]\L^{-nm\gamma+n}$ as well as for its image under $\Theta$.

\begin{proposition}\label{virignia}
\begin{itemize}
  \item[(i)] For $X$ as previous, $\widetilde\hl_m([X])=\loc(\widetilde{X}[m])$.
  \item[(ii)] Let $f$ be a nonzero function on a $d$-dimensional smooth connected $k$-variety $\mathscr X$, $\x$ a point of $f^{-1}(0)$. Let $\pi$ be the reduction map $\mathscr X(R)\to \mathscr X(\k)$. Set 
$$X:=\{x\in \mathscr X(R)\mid \pi(x)=\x, \rv(f(x))=\rv(t)\}.$$ 
Then $\hl([X])=\loc(\mathscr S_{f,\x})$.  
  \item[(iii)] For any $\gamma$ in $\Gamma$, $\hl([\gamma]_1)=1$ and $\hl(\overline{[\gamma]}_1)=\Lbb$. (Note that $[\gamma]_1$ and $\overline{[\gamma]}_1$ are the open and closed disks of valuative radius $\gamma$.)  
\end{itemize}  
\end{proposition}

\begin{proof}
(i) See Hrushovski-Loeser \cite{HL}.

(ii) We use \cite[Corollary 8.4.2]{HL} for proving (ii). Since $X$ is $2$-invariant (it is in fact $\gamma$-invariant for any $\gamma>1$ in $\Gamma$),  
\begin{align*}
X[m;2]=\left\{\varphi\in \mathscr X\Big(k[t^{1/m}]/(t^2)\Big)\mid \varphi(0)=\x, \rv(f(\varphi))=\rv(t)\right\}.
\end{align*}
The condition $\rv(f(\varphi))=\rv(t)$ is equivalent to $f(\varphi)\equiv t\mod t^{(m+1)/m}$, thus $X[m;2]$ is definably isomorphic via the map $t^{1/m}\mapsto t$ to 
$$\left\{\varphi\in \mathscr X\Big(k[t]/(t^{m+1})\Big)\mid \varphi(0)=\x, f(\varphi)\equiv t^m\mod t^{m+1}\right\}\times \mathbb A_k^{(m-1)d_1}.$$
We get $\widetilde\hl_m([X])=\loc([\mathscr X_{0,m}]\Lbb^{-md_1})$ and the conclusion follows.

(iii) Assume $\gamma=a/b$ with $a$, $b$ in $\mathbb Z$ and $(a,b)=1$. Then $\widetilde\hl_n([a/b]_1)=\Lbb^{-ma}$ if $n=mb$ and $\widetilde\hl_n([a/b]_1)=0$ otherwise, thus $\hl([a/b]_1)=1$. Also, $\widetilde\hl_n(\overline{[a/b]}_1)=\Lbb^{-ma+1}$ if $n=mb$ and $\widetilde\hl_{n}(\overline{[a/b]}_1)=0$ otherwise, thus $\hl(\overline{[a/b]}_1)=\Lbb$.
\end{proof}


\subsubsection{Formal case}
Let $X$ be a rigid $k((t))$-variety which is the generic fiber of a special formal $k[[t]]$-scheme $\mathfrak X$, let $\omega$ be a gauge form on $X$. We set $\overline{\mathfrak X}:=\mathfrak X\widehat{\otimes}_{k[[t]]}k[[t]]^{\alg}$ and $\overline{X}:=X\otimes_{k((t))}k((t))^{\alg}$. The integer-valued function $\ord_{\mathfrak X}(\omega)$ on $X$ was already recalled in (\ref{order}). Using the same way, one may define a rational-valued function $\ord_{\overline{\mathfrak X}}(\overline{\omega})$ on $\overline{X}$, where $\overline{\omega}$ is the pullback of $\omega$ via a the natural morphism $\overline{X}\to X$. We denote this rational-valued function by $\val_{\omega}$.

\begin{theorem}\label{comp}
Let $\mathfrak X$ be a relatively $d$-dimensional special formal $k[[t]]$-scheme with structural morphism $\mathfrak f$. Let $\overline{\mathfrak X}_{\eta,\rv}$ (resp. $\mathfrak X_{\eta}(m)_{\rv}$) be a version of $\overline{\mathfrak X}_{\eta}$ (resp. $\mathfrak X_{\eta}(m)$) in which $\mathfrak f_{\eta}(x)=t$ is replaced by $\rv\mathfrak f_{\eta}(x)=\rv(t)$ (resp. $\mathfrak f_{\eta}(x)\equiv t\mod t^{(m+1)/m}$). Then, for any gauge form $\omega$ on $\mathfrak X_{\eta}$, 
\begin{itemize}
  \item[(i)] $\hl_m([(\overline{\mathfrak X}_{\eta},\val_{\omega})])=\loc\Big(\Lbb^d\int_{\mathfrak X_{\eta}(m)}|\omega(m)|\Big)$,
  \item[(ii)] $\hl_m([(\overline{\mathfrak X}_{\eta,\rv},\val_{\omega})])=\loc\Big(\Lbb^d\int_{\mathfrak X_{\eta}(m)_{\rv}}|\omega(m)|\Big)$,
  \item[(iii)] $\hl([\overline{\mathfrak X}_{\eta,\rv}])=\hl([\overline{\mathfrak X}_{\eta}])$.
\end{itemize}
As a consequence, for a closed point $\x$ of $\mathfrak X_0$ and a gauge form $\omega'$ on $]\x[$, 
\begin{itemize}
  \item[(iv)] $\hl_m([(\overline{]\x[},\val_{\omega'})])=\loc\Big(\Lbb^d\int_{]\x[_m}|\omega'(m)|\Big)$,
  \item[(v)] $\hl_m([(\overline{]\x[}_{\rv},\val_{\omega'})])=\loc\Big(\Lbb^d\int_{]\x[_{m,\rv}}|\omega'(m)|\Big)$,
  \item[(vi)] $\hl(\overline{]\x[}_{\rv})=\hl(\overline{]\x[})=\loc\Big(\mathscr S_{\mathfrak f,\x}\Big)$.
\end{itemize}
\end{theorem}

\begin{proof}
We prove (i). By Lemma \ref{simp},
\begin{align}\label{000}
\hl_m([(\overline{\mathfrak X}_{\eta},\val_{\omega})])=\sum_{e\in\mathbb Z}\widetilde\hl_m([\val_{\omega}^{-1}(e/m)])\Lbb^{-e}.
\end{align}
By \cite[Lemma 3.1.1]{HL}, for each $e$ in $\mathbb Z$, there exists a $\gamma_{e,m}$ in $\Gamma$ such that $\val_{\omega}^{-1}(e/m)$ is $\gamma_{e,m}$-invariant. Thus it follows from Proposition \ref{virignia}(i) that 
\begin{equation}\label{aaa}
\begin{aligned}
\widetilde\hl_m([\val_{\omega}^{-1}(e/m)])&=\loc\Big(\widetilde{\val_{\omega}^{-1}(e/m)}[m]\Big)\\
&=\loc\Big([\val_{\omega}^{-1}(e/m)[m;\gamma']]\Lbb^{-md\gamma'+d}\Big)
\end{aligned}
\end{equation}
for any $\gamma'\geq \gamma_{e,m}$ in $(1/m)\mathbb Z\subset\Gamma$. 

Let $\mathfrak Y\to\mathfrak X$ be a N\'eron smoothening for $\mathfrak X$ with $\mathfrak Y$ relatively $d$-dimensional $k[[t]]$-formal scheme topologically of finite type. It is obvious that $\overline{\mathfrak X}_{\eta}=\overline{\mathfrak Y}_{\eta}$ since $k[[t]]$ is henselian, so we can regard $\val_{\omega}$ as a function on $\overline{\mathfrak Y}_{\eta}$. As the function $\val_{\omega}$ induces from the gauge form $\omega$, thus $\val_{\omega}^{-1}(e/m)(m)$ is a stable cylinder in $\Gr(\mathfrak Y(m))$, moreover, $\Big(\ord_{\mathfrak Y(m)}(\omega(m))\Big)^{-1}(e)=\val_{\omega}^{-1}(e/m)(m)$. By definition of the measure $\mu$, cf. (\ref{measure}), we have
\begin{equation}\label{bbb}
\begin{aligned}
\mu\Big(\Big(\ord_{\mathfrak Y(m)}(\omega(m))\Big)^{-1}(e)\Big)&=\mu\Big(\val_{\omega}^{-1}(e/m)(m)\Big)\\
&=[\val_{\omega}^{-1}(e/m)[m;\gamma']]\Lbb^{-md\gamma'}
\end{aligned}
\end{equation}
for $\gamma'$ in $\mathbb N$ large enough. From (\ref{000}), (\ref{aaa}) and (\ref{bbb}), it follows that
\begin{align*}
\hl_m([(\overline{\mathfrak X}_{\eta},\val_{\omega})])&=\loc\Big(\Lbb^d\int_{\mathfrak Y_{\eta}(m)}|\omega(m)|\Big)=\loc\Big(\Lbb^d\int_{\mathfrak X_{\eta}(m)}|\omega(m)|\Big).
\end{align*}
This identity is also compatible with the canonical $\mu_m$-action by definition, thus it holds in $\mathscr M_{k,\loc}^{\hat{\mu}}$. The identities (ii)-(vi) are direct consequences of the first one.
\end{proof}

\begin{remark}
In \cite{Thuong3}, we define the {\it motivic nearby cycles} of a formal function $\mathfrak f$ and denote it by $\mathscr S_{\mathfrak f}$. This is a virtual variety in the Grothendieck ring $\mathscr M_{\mathfrak X_0}^{\hat{\mu}}$ of $\mathfrak X_0$-varieties with good $\hat{\mu}$-action. In the context of Theorem \ref{comp}, (iii), the quantity $\hl([\overline{\mathfrak X}_{\eta}])$ is nothing but $\loc\left(\int_{\mathfrak X_0}\mathscr S_{\mathfrak f}\right)$, where $\int_{\mathfrak X_0}$ is the forgetful (or pushforward) morphism $\mathscr M_{\mathfrak X_0}^{\hat{\mu}}\to \mathscr M_k^{\hat{\mu}}$.
\end{remark}


\section{A new proof for the motivic Thom-Sebastiani theorem}\label{SSregular}
In this section, we give a model-theoretic proof for Theorem \ref{thom-sebastiani} by using the morphisms of rings $\widetilde\hl_m$ and $\hl$. For notational simplicity, we let $f$ and $g$ be regular functions on $\mathbb A_k^{d_1}$ and $\mathbb A_k^{d_2}$, vanishing at their origins, respectively. Then, we shall prove that the following identity holds in $\mathscr M_{k,\loc}^{\hat{\mu}}$:
\begin{align}\label{neww}
\loc\left(\mathscr S_{f\oplus g,(0,0)}\right)=\loc\left(-\mathscr S_{f,0}* \mathscr S_{g,0}+\mathscr S_{f,0}+\mathscr S_{g,0}\right)
\end{align}

\subsection{Decomposition of the analytic Milnor fiber}
Consider the analytic Minor fiber of $f\oplus g$ at the origin of $\A_k^{d_1}\times\mathbb A_k^{d_2}$, 
$$Z:=\left\{(x,y)\in \m^{d_1+d_2}\mid \rv(f(x)+g(y))=\rv(t)\right\}.$$
This is a bounded $2$-invariant definable subset of $\VF^{d_1+d_2}$. By Proposition \ref{virignia}(ii), $\hl([Z])=\loc\left(\mathscr S_{f\oplus g, (0,0)}\right)$ in $\mathscr M_{ k,\loc}^{\hat{\mu}}$. Let us decompose $Z$ into a disjoint union of sets $X$, $Y$ and $Z^*$ subject to conditions $\val f(x)<\val g(y)$, $\val f(x)>\val g(y)$ and $\val f(x)=\val g(y)$, respectively. In the sequel, we are going to compute $\hl([X])$, $\hl([Y])$, $\hl([Z^*])$ and conclude.

Write $X=\{(x,y)\in \m^{d_1+d_2}\mid \rv(f(x))=\rv(t)\}$ as the product of the definable sets $X':=\{x\in \m^{d_1}\mid\rv f(x)=\rv(t)\}$ and $\m^{d_2}=[0]_1^{d_2}$. Proposition \ref{virignia}, the items (ii) and (iii), gives $\hl([X_1])=\loc\left(\mathscr S_{f,0}\right)$ and $\hl([\m^{d_2}])=1$, thus $\hl([X])=\loc\left(\mathscr S_{f,0}\right)$ in $\mathscr M_{k,\loc}^{\hat{\mu}}$. Similarly, we also have $\hl([Y])=\loc\left(\mathscr S_{g,0}\right)$ in $\mathscr M_{k,\loc}^{\hat{\mu}}$.

Set $Z_1^*=\{(x,y)\in Z^* \mid \val f(x)=1\}$, $Z_{<1}^*=\{(x,y)\in Z^* \mid 0<\val f(x)<1\}$, then $Z^*=Z_1^*\sqcup Z_{<1}^*$. For our goal we introduce the following definable set
\begin{align*}
Z_0:=\left\{(x,y)\in \m^{d_1+d_2}\mid \val(f(x)+g(y))>1, -\rv f(x)=\rv g(y)=\rv(t)\right\}.
\end{align*}
We shall consider the identity $[Z^*]=\left([Z_1^*]-[Z_0]\right)+\left([Z_{<1}^*]+[Z_0]\right)$ in $K(\vol\VF^{\bdd})$.

\begin{proposition}\label{-1-time}
For $m\geq 1$, the equality 
$$\widetilde\hl_m\left([Z_1^*]-[Z_0]\right)=-\loc\left([\mathscr X_{0,m}(f)]*[\mathscr X_{0,m}(g)]\L^{-m(d_1+d_2)}\right)$$
holds in $\mathscr M_{k,\loc}^{\hat{\mu}}$. Moreover, also in this ring $\mathscr M_{k,\loc}^{\hat{\mu}}$, we have
$$\hl\left([Z_1^*]-[Z_0]\right)=-\loc\left(S_{f,0}*S_{g,0}\right).$$
\end{proposition}

\begin{lemma}\label{ICTP}
The following hold in $\mathscr M_{k,\loc}^{\hat{\mu}}$:
\begin{itemize}
  \item[(i)] $\widetilde\hl_m([Z_1^*])=\loc\left([\mathscr X_{m,0}(f)\times \mathscr X_{m,0}(g)\times^{\mu_m\times \mu_m} F_1^m]\Lbb^{-m(d_1+d_2)}\right)$;
  \item[(ii)] $\widetilde\hl_m([Z_0])=\loc\left([\mathscr X_{m,0}(f)\times \mathscr X_{m,0}(g)\times^{\mu_m\times \mu_m} F_0^m]\Lbb^{-m(d_1+d_2)}\right)$.
\end{itemize}
\end{lemma}

\begin{proof}
(i) Since $Z_1^*$ is $2$-invariant, we consider $Z_1^*[m;2]$ which equals 
\begin{align*}
&\left\{(\varphi,\psi)\in\left(\frac{k[t^{1/m}]}{t^2}\right)^{d_1+d_2}
\begin{array}{|l}
(\varphi(0),\psi(0))=(0,0), \val f(\varphi)=\val g(\psi)=1\\
f(\varphi)+g(\psi)\equiv t\mod t^{(m+1)/m}
\end{array}
\right\}\\
&\cong\left\{(\varphi,\psi)\in\left(\frac{tk[t]}{t^{2m}}\right)^{d_1+d_2}
\begin{array}{|l}
\ord f(\varphi)=\ord g(\psi)=m\\
f(\varphi)+g(\psi)\equiv t^m\mod t^{m+1}
\end{array}
\right\}\\
&\cong
\left\{(\varphi,\psi)\in\left(\frac{tk[t]}{t^{m+1}}\right)^{d_1+d_2}
\begin{array}{|l}
\ord f(\varphi)=\ord g(\psi)=m\\
f(\varphi)+g(\psi)\equiv t^m\mod t^{m+1}
\end{array}
\right\} \times \A_k^{(m-1)(d_1+d_2)}.
\end{align*}
We claim that there is a canonical isomorphism between 
$$V:=\left\{(\varphi,\psi)\in\left(\frac{tk[t]}{t^{m+1}}\right)^{d_1+d_2}
\begin{array}{|l}
\ord f(\varphi)=\ord g(\psi)=m\\
f(\varphi)+g(\psi)\equiv t^m\mod t^{m+1}
\end{array}
\right\}
$$
and
$$\mathscr X_{0,m}(f)\times \mathscr X_{0,m}(g)\times^{\mu_m\times \mu_m} F_1^m.$$
Indeed, we define a map $\mathscr X_{0,m}(f)\times \mathscr X_{0,m}(g)\times F_1^m\to V$ that sends $(\varphi(t),\psi(t); a,b)$ to $(\varphi(at),\psi(bt))$. It then induces a well defined morphism on the quotient 
$$\xi: \mathscr X_{0,m}(f)\times \mathscr X_{0,m}(g)\times^{\mu_m\times \mu_m} F_1^m\to V.$$
We also define a morphism 
$$\eta: V\to \mathscr X_{0,m}(f)\times \mathscr X_{0,m}(g)\times^{\mu_m\times \mu_m} F_1^m$$
given by $\eta(\varphi(t),\psi(t))=(\varphi((\ac f\varphi)^{-1/m}t),\psi((\ac g\psi)^{-1/m}t);(\ac f\varphi)^{1/m}, (\ac g\psi)^{1/m})$. It is clear that $\xi$ and $\eta$ are inverse of each other and the claim follows. Consequently, by Proposition \ref{virignia}(i), 
\begin{align*}
\widetilde\hl_m([Z_1^*])=\loc\left([\mathscr X_{0,m}(f)\times \mathscr X_{0,m}(g)\times^{\mu_m\times \mu_m} F_1^m]\L^{-m(d_1+d_2)}\right).
\end{align*}

(ii) Similarly as previous, since $Z_0$ is $2$-invariant, $Z_0[m;2]$ is isomorphic to
\begin{align*}
\left\{(\varphi,\psi)\in\left(\frac{tk[t]}{t^{m+1}}\right)^{d_1+d_2}
\begin{array}{|l}
\ord\left(f(\varphi)+g(\psi)\right)>m\\
-f(\varphi)\equiv g(\psi)\equiv t^m\mod t^{m+1}
\end{array}
\right\}\times \A_k^{(m-1)(d_1+d_2)}
\end{align*}
Also as above, we are able to prove that the constructible set 
$$\left\{(\varphi,\psi)\in\left(\frac{tk[t]}{t^{m+1}}\right)^{d_1+d_2}
\begin{array}{|l}
\ord\left(f(\varphi)+g(\psi)\right)>m\\
-f(\varphi)\equiv g(\psi)\equiv t^m\mod t^{m+1}
\end{array}
\right\}
$$
is isomorphic to $\mathscr X_{0,m}(f)\times \mathscr X_{0,m}(g)\times^{\mu_m\times \mu_m} F_0^m$, thus (ii) is proven.
\end{proof}

\begin{proof}[Proof of Proposition \ref{-1-time}]
By Lemma \ref{ICTP} and by definition of convolution product (cf. Subsection \ref{TSconv}) we get $\widetilde\hl_m\left([Z_1^*]-[Z_0]\right)=-\loc\left([\mathscr X_{0,m}(f)]*[\mathscr X_{0,m}(g)]\L^{-m(d_1+d_2)}\right)$. By a property of the Hadamard product, namely,
\begin{align*}
-&\lim_{T\to\infty}\sum_{m\geq 1}-[\mathscr X_{0,m}(f)]*[\mathscr X_{0,m}(g)]\L^{-m(d_1+d_2)}T^m\\
&=-\left(-\lim_{T\to\infty}\sum_{m\geq 1}[\mathscr X_{0,m}(f)]\L^{-md_1}T^m\right)*\left(-\lim_{T\to\infty}\sum_{m\geq 1}[\mathscr X_{0,m}(g)]\L^{-md_2}T^m\right)\\
&=-S_{f,0}*S_{g,0},
\end{align*}
we deduce that $\hl\left([Z_1^*]-[Z_0]\right)=-\loc\left(S_{f,0}*S_{g,0}\right)$ in $\mathscr M_{k,\loc}^{\hat{\mu}}$.
\end{proof}



\subsection{Integral over $\Gamma$}
Let $D$ be a definable subset of $\Gamma$. A function $\nu:D\to \mathscr M_{k,\loc}^{\hat{\mu}}$ is called {\it definable} if $D$ may be partitioned into finitely many disjoint definable subsets $D_i$, $i\in I$, such that $\nu|_{D_i}$ is constant $c_i\in \mathscr M_{k,\loc}^{\hat{\mu}}$ for every $i$ in $I$. Then, we define the integral of $\nu$ over $D$, which takes value in $\mathscr M_{k,\loc}^{\hat{\mu}}$, as follows
\begin{align*}
\int_{\gamma\in D}\nu(\gamma)=\int_{\gamma\in D}\nu(\gamma)d\chi:=\sum_{i\in I}c_i\chi(D_i).
\end{align*}
Here, $\chi$ is the o-minimal Euler characteristic defined in \cite[Lemma 9.5]{HK} followed by the localization morphism. 


\subsection{Completion of the proof of Theorem \ref{thom-sebastiani}}\label{subsec5.3}
In this subsection, we shall prove that $\hl\left([Z_{<1}^*]+[Z_0]\right)=0$ in $\mathscr M_{k,\loc}^{\hat{\mu}}$, thus finish the proof of (\ref{neww}). 

\subsubsection{Computation of $\hl([Z_{<1}^*])$ and $\hl([Z_0])$}
Let $\pi_{<1}$ denotes the definable function $Z_{<1}^*\to (0,1)\subset \Gamma$ mapping $(x,y)$ to $\val f(x)$, and let $\nu: (0,1)\to \mathscr M_{k,\loc}^{\hat{\mu}}$ be the function defined by 
\begin{align*}
\nu(\gamma)=\hl([\pi_{<1}^{-1}(\gamma)]). 
\end{align*}

\begin{lemma}\label{0-time}
The function $\nu$ is definable.
\end{lemma}

\begin{proof}
Via the definable bijection $(x,y)\mapsto (x,y,\val f(x))$, we may regard $Z_{<1}^*$ as a definable subset of $\m^{d_1+d_2}\times (0,1)$. Consider the surjective morphism of rings 
$$N_0:K(\vol\Gamma^{\bdd})\otimes K(\vol\RES)\to K(\vol\VF^{\bdd})$$
induced by $N$ in (\ref{morphismN}). There exist definable subsets $W_i$ of $\RES^i$ and $\Delta_i$ of $\Gamma^i\times (0,1)$, $0\leq i\leq d_1+d_2$, with $N_0\left(\sum_{i=0}^{d_1+d_2}[\Delta_i]\otimes [W_{d_1+d_2-i}]\right)=[Z_{<1}^*]$. By definition of $\alpha$, $\beta$ (cf. \ref{4.4.2}), $(\alpha\otimes\beta)\left(\sum_{i=0}^{d_1+d_2}[\Delta_i]\otimes [W_{d_1+d_2-i}]\right)=\sum_{i=0}^{d_1+d_2}\chi(\Delta_i)w_{d_1+d_2-i}$, where $w_{d_1+d_2-i}:=[W_{d_1+d_2-i}](\Lbb-1)^i$. Similarly, for $\gamma\in(0,1)$, there are definable subsets $W_{\gamma,i}$ of $\RES^i$, $\Delta_{\gamma,i}$ of $\Gamma^i\times \{\gamma\}$ with $N_0\left(\sum_{i=0}^{d_1+d_2}[\Delta_{\gamma,i}]\otimes [W_{\gamma,d_1+d_2-i}]\right)=[\pi_{<1}^{-1}(\gamma)]$. Also, $(\alpha\otimes\beta)\left(\sum_{i=0}^{d_1+d_2}[\Delta_{\gamma,i}]\otimes [W_{\gamma,d_1+d_2-i}]\right)=\sum_{i=0}^{d_1+d_2}\chi(\Delta_{\gamma,i})w_{\gamma,d_1+d_2-i}$, where $w_{\gamma,d_1+d_2-i}:=[W_{\gamma,d_1+d_2-i}](\Lbb-1)^i$. We claim that $w_i=w_{\gamma,i}$ in $!K(\RES)$. Indeed, the image of $W_i$ (resp. $W_{\gamma,i}$) in $K(\vol\VF^{\bdd})$ is $[W_i\times_{\ell,\res}R^i]$ (resp. $[W_{\gamma,i}\times_{\ell_{\gamma},\res}R^i]$), where $\ell:W_i\to\k^i$ and $\ell_{\gamma}:W_{\gamma,i}\to\k^i$ are \'etale maps, $R=\{\tau\in\VF\mid \val(\tau)\geq 0\}$. The unique difference between $W_i\times_{\ell,\res}R^i$ and $W_{\gamma,i}\times_{\ell_{\gamma},\res}R^i$ is that the former admits the condition $0<\val f(x)<1$ while the latter satisfies $\val f(x)=\gamma$. Thus $[W_i]=[W_{\gamma,i}]$ in $!K(\RES)$. Consequently, $(\alpha\otimes\beta)\left(\sum_{i=0}^{d_1+d_2}[\Delta_{\gamma,i}]\otimes [W_{\gamma,d_1+d_2-i}]\right)=\sum_{i=0}^{d_1+d_2}\chi(\Delta_{\gamma,i})w_{d_1+d_2-i}$. Since $\hl$ induces from $\alpha\otimes \beta$, we have the following
\begin{align}\label{godent}
\hl([Z_{<1}^*])=\sum_{i=0}^{d_1+d_2}\chi(\Delta_i)\Theta_i,\quad \hl([\pi_{<1}^{-1}(\gamma)])=\sum_{i=0}^{d_1+d_2}\chi(\Delta_{\gamma,i})\Theta_i,
\end{align}
where $\Theta_i:=\Theta(w_{d_1+d_2-i})$.

For $0\leq i\leq d_1+d_2$, we identify $\Gamma^i\times(0,1)$ with a subset of $\Gamma^{d_1+d_2}\times (0,1)$ in an obvious manner. Let $\pr_2$ be the second projection $\Gamma^{d_1+d_2}\times (0,1)\to (0,1)$ and $D_i:=\pr_2(\Delta_i)$. Then $(0,1)=\bigsqcup_{i=0}^{d_1+d_2}D_i$. Moreover, for any $\gamma$ in $(0,1)$, $\Delta_{\gamma,i}$ is a fiber of the definable map $\Delta_i\to D_i$, all the fibers of this map are definably isomorphic. Thus $\chi(\Delta_i)=\chi(D_i)\chi(\Delta_{\gamma,i})$. It and (\ref{godent}) show that, on $D_i$,
\begin{align}\label{Vienna2}
\nu(\gamma)=\hl([\pi_{<1}^{-1}(\gamma)])=\sum_{i=0}^{d_1+d_2}\chi(\Delta_i)\chi(D_i)^{-1}\Theta_i,
\end{align}
which proves the definability.
\end{proof}

\begin{corollary}\label{first-time}
$\int_{\gamma\in (0,1)}\hl([\pi_{<1}^{-1}(\gamma)])=\hl([Z_{<1}^*])$.
\end{corollary}

\begin{proof}
By definition as well as by (\ref{godent}) and (\ref{Vienna2}),  
\begin{align*}
\int_{\gamma\in (0,1)}\nu(\gamma)=\sum_{i=0}^{d_1+d_2}\nu|_{D_i}\chi(D_i)=\sum_{i=0}^{d_1+d_2}\chi(\Delta_i)\chi(D_i)^{-1}\Theta_i\chi(D_i)=\hl([Z_{<1}^*]).
\end{align*}
\end{proof}

Let $\pi_0$ be the function $Z_0\to (1,\infty)\subset \Gamma$ that sends $(x,y)$ to $\val(f(x)+g(y))$. Similarly as previous, we are able to prove the following 

\begin{corollary}
The function $(1,\infty)\to \mathscr M_{k,\loc}^{\hat{\mu}}$ given by $\hl([\pi_0^{-1}(\gamma)])$ is definable, moreover,
$\int_{\gamma\in (1,\infty)}\hl([\pi_0^{-1}(\gamma)])=\hl([Z_0])$.
\end{corollary}

\subsubsection{Conclusion}
Let $A$ be the annulus $\{\tau\in \VF\mid 0<\val(\tau)<1\}$ and $p_{<1}$ the function $Z_{<1}^*\to A$ mapping $(x,y)$ to $f(x)$. Then $\pi_{<1}=p_{<1}\circ\val$. The fiber over $\tau\in A$ of $p_{<1}$ is the following    
\begin{align}\label{Vienna3}
p_{<1}^{-1}(\tau)=\left\{(x,y)\in \m^{d_1+d_2}\mid f(x)=\tau, g(y)=-\tau+t\right\}.
\end{align}
As for each $\gamma$ in $(0,1)$, all the fibers $p_{<1}^{-1}(\tau)$, $\tau$ in $\val^{-1}(\gamma)$, are definably isomorphic, since the description (\ref{Vienna3}), it implies that the equalities
$$[\pi_{<1}^{-1}(\gamma)]=\int_{\tau\in \val^{-1}(\gamma)}[p_{<1}^{-1}(\tau)]=[\val^{-1}(\gamma)][p_{<1,\gamma}^{-1}]$$
hold in $K(\vol\VF^{\bdd})$, where $[p_{<1,\gamma}^{-1}]$ is the constant function $[p_{<1}^{-1}(\tau)]$ on $\val^{-1}(\gamma)$. By Corollary \ref{first-time},
\begin{align}\label{Vienna4}
\hl([Z_{<1}^*])=\int_{\gamma\in (0,1)}\hl([\pi_{<1}^{-1}(\gamma)])=\int_{\gamma\in (0,1)}\hl([\val^{-1}(\gamma)])\hl([p_{<1,\gamma}^{-1}]).
\end{align}

\begin{lemma}\label{second-time}
$\hl([p_{<1,\gamma}^{-1}])$ is independent of the choice of $\gamma$ in $(0,1)$.
\end{lemma}

\begin{proof}
Using (\ref{Vienna3}), namely, 
$$p_{<1}^{-1}(\tau)=\left\{x\in \m^{d_1}\mid f(x)=\tau\right\}\times \left\{y\in\mathfrak m^{d_2}\mid g(y)=-\tau+t \right\},$$ 
it suffices to prove $\hl([\{x\in \m^d\mid f(x)=t^{\gamma}\}])=\hl([\{x\in \m^d\mid f(x)=t\}])$ for any regular function $f:\A_k^d\to \A_k^1$ vanishing at the origin of $\A_k^d$ and for any $\gamma$ in $(0,1)$. 
Equivalently, it suffices to prove $\hl([\{x\in \m^d\mid \rv f(x)=\rv(t^{\gamma})\}])=\hl([\{x\in \m^d\mid \rv f(x)=\rv(t)\}])$ for $\gamma=a/b$ in $(0,1)$ with $a$ and $b$ coprime integers, $a<b$. Indeed, if $m$ is not divisible by $b$, then $\hl_m([\{x\in \m^d\mid \rv f(x)=\rv(t^{a/b})\}])=0$. Otherwise, say, $m=bs$, then  
\begin{align*}
\widetilde\hl_{bs}([\{x\in \m^d\mid \rv f(x)=\rv(t^{a/b})\}])=\widetilde\hl_{as}([\{x\in \m^d\mid \rv f(x)=\rv(t)\}]),
\end{align*}
because, by a simple geometric computation, both sides are equal to $[\mathscr X_{0,as}(f)]\L^{-asd}$. This equality then implies the lemma.
\end{proof}

Using (\ref{Vienna4}) and Lemma \ref{second-time}, we get $\hl([Z_{<1}^*])=\left(\int_{\gamma\in (0,1)}\hl([\val^{-1}(\gamma)])\right)\hl([p_{<1,\gamma}^{-1}])$. Similarly as in the proofs of Lemma \ref{0-time} and Corollary \ref{first-time}, we can easily show that $\int_{\gamma\in (0,1)}\hl([\val^{-1}(\gamma)])=\hl([A])=-1$. Thus 
\begin{align}\label{Vienna5}
\hl([Z_{<1}^*])=-\hl([p_{<1,\gamma}^{-1}]) \quad (\gamma\in (0,1)).
\end{align} 

Denote by $B$ the set $\{\tau\in \VF\mid \val(\tau)>1\}$ and consider the function $p_0: Z_0\to B$ defined by $p_0(x,y)=f(x)+g(y)$. Then, we have $\pi_0=p_0\circ\val$, moreover, the fiber over $\tau\in B$ of $p_0$ equals 
\begin{align*}
p_0^{-1}(\tau)&=\left\{(x,y)\in \m^{d_1+d_2} \mid f(x)+g(y)=\tau, -\rv f(x)=\rv g(y)=\rv(t)\right\}\\
&=\left\{(x,y)\in \m^{d_1+d_2} \mid f(x)=ct, g(y)=-ct+\tau, c \in 1+\mathfrak m\right\}.
\end{align*}
Similarly as in the proof of Lemma \ref{second-time}, we can show that $\hl([p_0^{-1}(\tau)])$ is independent of $\tau \in B$ and, moreover, that
\begin{align}\label{Vienna9}
\hl([p_{<1,\gamma}^{-1}])=\hl([p_0^{-1}(\tau)])
\end{align}
for any $\gamma$ in $(0,1)$ and any $\tau$ in $B$. An analogue of Lemma \ref{0-time} and Corollary \ref{first-time} gives rise to the formula
\begin{align}\label{Vienna8}
\hl([Z_0])=\hl([B])\hl([p_0^{-1}(\tau)])=\hl([p_0^{-1}(\tau)])\quad (\tau\in B).
\end{align}
Finally, it follows from (\ref{Vienna5}), (\ref{Vienna8}) and (\ref{Vienna9}) that $\hl\left([Z_{<1}^*]+[Z_0]\right)=0$ in $\mathscr M_{k,\loc}^{\hat{\mu}}$. This together with Proposition \ref{-1-time} proves (\ref{neww}).


\section{Proof of Theorem \ref{TSformal}}\label{CC}
It is Theorem \ref{comp} that completely interprets the role of the morphisms $\hl_m$ and $\hl$ in understanding the motivic Milnor fiber of a formal function from the non-archimedean geometry point of view. Motivated by this, to prove Theorem \ref{TSformal}, also as the proof of the regular version (Section \ref{SSregular}), we work on analytic Milnor fibers (in the sense of \cite{NS}) considered as definable sets in the theory $\ACVF_{k((t))}(0,0)$.

\subsection{Using arguments in Section \ref{SSregular}}\label{subsec6.1}
Let $\Z$ be the analytic Minor fiber $\overline{](0,0)[}$ of $\mathfrak f\oplus \mathfrak g$ at the origin $(0,0)$ of $\A_k^{d_1}\times\mathbb A_k^{d_2}$, namely, 
$$\Z=\left\{(x,y)\in \m^{d_1+d_2}\mid f(x)+g(y)=t\right\}.$$
(To indicate precisely the origin of $\mathbb A_k^{d_i}$, if necessary, we write $0_{d_i}$ instead of $0$.) It induces immediately from Theorem \ref{comp} that $\hl([\Z])=\loc\left(\mathscr S_{\mathfrak f\oplus\mathfrak g, (0,0)}\right)$ in $\mathscr M_{ k,\loc}^{\hat{\mu}}$. Write $\Z$ as a disjoint union of definable subsets $\X$, $\Y$ and $\Z^*$ respectively defined by $\val f(x)<\val g(y)$, $\val f(x)>\val g(y)$ and $\val f(x)=\val g(y)$. Also by Theorem \ref{comp}, we have $\hl([\X])=\loc\left(\mathscr S_{\mathfrak f,0}\right)$ and $\hl([\Y])=\loc\left(\mathscr S_{\mathfrak g,0}\right)$ in $\mathscr M_{k,\loc}^{\hat{\mu}}$.

To continue, we modify slightly $\Z^*$ into $Z^{\fl}$, where
$$Z^{\fl}=\left\{(x,y)\in \m^{d_1+d_2}\mid \rv(f(x)+g(y))=\rv(t), \val f(x)=\val g(y)\right\},$$
and note that $\hl([Z^{\fl}])=\hl([0]_1^{d_1+d_2}\cdot[\Z^*])=\hl([\Z^*])$ in $\mathscr M_{k,\loc}^{\hat{\mu}}$ since $\hl([0]_1)=1$. Now, decompose $Z^{\fl}$ into a disjoint union of $Z_1^{\fl}=\{(x,y)\in Z^{\fl} \mid \val f(x)=1\}$ and $Z_{<1}^{\fl}=\{(x,y)\in Z^{\fl} \mid 0<\val f(x)<1\}$. Similarly as in Section \ref{SSregular}, we use the definable set 
$$Z_0^{\fl}:=\left\{(x,y)\in \m^{d_1+d_2}\mid \val(f(x)+g(y))>1, -\rv f(x)=\rv g(y)=\rv(t)\right\}$$ 
and present $[Z^{\fl}]$ as the sum $([Z_1^{\fl}]-[Z_0^{\fl}])+([Z_{<1}^{\fl}]+[Z_0^{\fl}])$ in $K(\vol\VF^{\bdd})$. As in Subsection \ref{subsec5.3}, we also obtain $\hl([Z_{<1}^{\fl}]+[Z_0^{\fl}])=0$. In the sequel, we shall prove that $\hl([Z_1^{\fl}]-[Z_0^{\fl}])=-\loc\left(\mathscr S_{\mathfrak f,0}* \mathscr S_{\mathfrak g,0}\right)$ and the proof of Theorem \ref{TSformal} is completed.

\subsection{Using motivic integral via Subsections \ref{subsec-LS}, \ref{HLconst}, \ref{action}, \ref{subsec-descriptions}, Section \ref{SSregular}} 
In this subsection, we prove the following
\begin{proposition}\label{DTCC}
With the previous notation, we have 
$$\hl([Z_1^{\fl}]-[Z_0^{\fl}])=-\loc\left(\mathscr S_{\mathfrak f,0}* \mathscr S_{\mathfrak g,0}\right).$$
\end{proposition}

Let $\mathfrak Z_1$ (resp. $\mathfrak Z_2$) be a N\'eron smoothening for the formal completion of $\mathfrak X$ at $0_{d_1}$ (resp. for the formal completion of $\mathfrak Y$ at $0_{d_2}$) with $\mathfrak Z_1$ and $\mathfrak Z_2$ smooth, topologically of finite type over $k[[t]]$. For any integer $m\geq 1$, let $\Gr(\mathfrak Z(m))_{\rv}$ be the space defined as $\Gr(\mathfrak Z(m))$ but $\mathfrak f(x)=t$ replaced by $\mathfrak f(x)\equiv t\mod t^{(m+1)/m}$, where $\mathfrak f$ is the structural morphism of $\mathfrak Z$.  For $i\in\{1, 2\}$, let $\omega_i$ be a bounded gauge form on $]0_{d_i}[$ (remark that $]0_{d_1}[=\mathfrak Z_{1,\eta}$ and $]0_{d_2}[=\mathfrak Z_{2,\eta}$), and, for any integer $e_i$, set 
\begin{align*}
\Phi(\mathfrak Z_i(m),\omega_i(m),e_i):&=\mu(\{\varphi\in\Gr(\mathfrak Z_i(m))_{\rv}\mid \ord_{\mathfrak Z_i(m)}(\omega_i(m))(\varphi)=e_i\}),
\end{align*}
which is an element of $\mathscr M_k^{\hat{\mu}}$, by the $\mu_m$-action $a\varphi(t):=\varphi(at)$. By definition,
$$\int_{]0_{d_i}[_{m,\rv}}|\omega_i(m)|=\sum_{e_i\in\mathbb Z}\Phi(\mathfrak Z_i(m),\omega_i(m),e_i)\Lbb^{-e_i},$$
for $i\in\{1, 2\}$, where the sum runs over a finite set as $\omega_i$ is a gauge form (see \cite{LS}). One thus deduces that 
\begin{equation}\label{QTLD}
\begin{aligned}
&\left(\int_{]0_{d_1}[_{m,\rv}}|\omega_1(m)|\right)*\left(\int_{]0_{d_2}[_{m,\rv}}|\omega_2(m)|\right)\\
&\qquad=\sum_{e_1,e_2\in\mathbb Z}\Phi(\mathfrak Z_1(m),\omega_1(m),e_1)*\Phi(\mathfrak Z_2(m),\omega_2(m),e_2)\Lbb^{-(e_1+e_2)}.
\end{aligned}
\end{equation}

For $e_1$, $e_2$ in $\Gamma$, let $Z_{1,e_1,e_2}^{\fl}$ (resp. $Z_{0,e_1,e_2}^{\fl}$) be the subset of $Z_1^{\fl}$ (resp. $Z_0^{\fl}$) such that $\val_{\omega_1}(x)=e_1$ and $\val_{\omega_2}(y)=e_2$. For $e$ in $\Gamma$, set $Z_{1,e}^{\fl}:=\bigcup_{e_1+e_2=e}Z_{1,e_1,e_2}^{\fl}$ and $Z_{0,e}^{\fl}:=\bigcup_{e_1+e_2=e}Z_{0,e_1,e_2}^{\fl}$.


\begin{lemma}\label{lem6.1}
For any integer $m\geq 1$, for any $e_1$, $e_2$ in $\Gamma$ with $me_1, me_2\in \mathbb Z$, 
\begin{align*}
&\widetilde\hl_m([Z_{1,e_1,e_2}^{\fl}]-[Z_{0,e1,e_2}^{\fl}])\\
&\quad=-\loc\left(\Lbb^{d_1+d_2}\Phi(\mathfrak Z_1(m),\omega_1(m),me_1)*\Phi(\mathfrak Z_2(m),\omega_2(m),me_2)\right).
\end{align*}
\end{lemma}

\begin{proof}
Since $\mathfrak Z_1$ is topologically of finite type, there exist a convergent power series $\widetilde f$ in $k\{x\}$ vanishing at $0$ (hence a $k[[t]]$-scheme $\mathscr X=\Spec\Big(k[[t]][x]/(\widetilde f(x)-t)\Big)$) such that  
\begin{align*}
\Gr_{\ell}(\mathfrak Z_1)(k)&=\left\{\varphi\in \left(\mathscr X\otimes_{k[[t]]}(k[t]/t^{\ell+1})\right)(k[t]/t^{\ell+1})\mid \val(\varphi)>0\right\}\\
&\cong \left\{\varphi\in \left(tk[t]/t^{\ell+1}\right)^{d_1}\mid \widetilde f(\varphi)=t\right\}
\end{align*}
for $\ell$ in $\mathbb N$. Similarly, $\Gr_{\ell}(\mathfrak Z_2)(k)\cong \{\varphi\in \left(tk[t]/t^{\ell+1}\right)^{d_2}\mid \widetilde g(\varphi)=t\}$ for some $\widetilde g$ in $k\{y\}$ with $\widetilde g(0)=0$. Thus, $\Lbb^{d_1}\Phi(\mathfrak Z_1(m),\omega_1(m),me_1)$ equals $\Lbb^{-\ell md_1}$ times
\begin{align*}
\left[\left\{\varphi\in \left(tk[t]/t^{\ell m+1}\right)^{d_1}\mid \widetilde f(\varphi)\equiv t^m\mod t^{m+1}, \ord_{\mathfrak Z_1}(\omega_1)(\varphi)=me_1\right\}\right],
\end{align*}
and $\Lbb^{d_2}\Phi(\mathfrak Z_2(m),\omega_2(m),me_2)$ equals $\Lbb^{-\ell md_2}$ times
\begin{align*}
\left[\left\{\psi\in \left(tk[t]/t^{\ell m+1}\right)^{d_2}\mid \widetilde g(\psi)\equiv t^m\mod t^{m+1}, \ord_{\mathfrak Z_2}(\omega_2)(\psi)=me_2\right\}\right],
\end{align*}
for $\ell\in\mathbb N$ large enough. At this time, we may use the arguments in the proof of Lemma \ref{ICTP} and Proposition \ref{-1-time}, hence conclusion.
\end{proof}

\begin{lemma}\label{lem6.2}
For any integer $m\geq 1$,
\begin{align*}
-\hl_m([(\overline{]0_{d_1}[}_{\rv},&\val_{\omega_1})])*\hl_m([(\overline{]0_{d_2}[}_{\rv},\val_{\omega_2})])\\
&=\hl_m([(Z_1^{\fl},\val_{\omega_1}\oplus\val_{\omega_2})]-[(Z_0^{\fl},\val_{\omega_1}\oplus\val_{\omega_2})]),
\end{align*}
where, by definition, $\val_{\omega_1}\oplus\val_{\omega_2}(x,y)=\val_{\omega_1}(x)+\val_{\omega_2}(y)$. 
\end{lemma}

\begin{proof}
Applying (\ref{QTLD}) and Lemmas \ref{lem6.1}, \ref{simp}, we get  
\begin{align*}
-&\hl_m([(\overline{]0_{d_1}[}_{\rv},\val_{\omega_1})])*\hl_m([(\overline{]0_{d_2}[}_{\rv},\val_{\omega_2})])\\
&=-\loc\left(\Lbb^{d_1}\int_{]0_{d_1}[_{m,\rv}}|\omega_1(m)|\right)*\loc\left(\Lbb^{d_2}\int_{]0_{d_2}[_{m,\rv}}|\omega_2(m)|\right)\\
&=-\loc\left(\Lbb^{d_1+d_2}\sum_{e_1,e_2\in\mathbb Z}\Phi(\mathfrak Z_1(m),\omega_1(m),e_1)*\Phi(\mathfrak Z_2(m),\omega_2(m),e_2)\Lbb^{-(e_1+e_2)}\right)\\
&=\sum_{e\in\mathbb Z}\widetilde\hl_m([Z_{1,e/m}^{\fl}]-[Z_{0,e/m}^{\fl}])\Lbb^{-e}\\
&=\hl_m([(Z_1^{\fl},\val_{\omega_1}\oplus\val_{\omega_2})]-[(Z_0^{\fl},\val_{\omega_1}\oplus\val_{\omega_2})]).
\end{align*}
The lemma is proven.
\end{proof}

\begin{proof}[Proof of Proposition \ref{DTCC} and Theorem \ref{TSformal}]
Thanks to Lemma \ref{lem6.2}, Proposition \ref{8May} and Theorem \ref{comp}, we have $\hl([Z_1^{\fl}]-[Z_0^{\fl}])=-\loc\left(\mathscr S_{\mathfrak f,0}* \mathscr S_{\mathfrak g,0}\right)$ as desired. The proof of Theorem \ref{TSformal} is deduced from Proposition \ref{DTCC} and Subsection \ref{subsec6.1}.
\end{proof}


\end{document}